\documentclass[preprint,12pt]{elsarticle}

\usepackage{amsthm}
\usepackage{amsmath,amssymb,txfonts}
\usepackage{color,bm,comment}
\usepackage{comment}
\usepackage[english]{babel}
\usepackage{array}

\newtheorem{theorem}{Theorem}
\newtheorem{lemma}{Lemma}
\newtheorem{prop}{Proposition}
\newtheorem{cor}{Corollary}
\newtheorem{remark}{Remark}
\newtheorem{example}{Example}

\theoremstyle{definition}
\newtheorem{definition}[theorem]{Definition}

\def\re{\mathbb{R}}
\def\N{\mathbb{N}}

\def\({\left(}
\def\){\right)}
\def\[{\left[}
\def\]{\right]}
\def\pd{\partial}
\def\lap{\Delta}
\def\dis{\displaystyle}
\def\ep{\varepsilon}
\def\w{\omega}
\def\la{\lambda}

\newtheorem{ThmA}{Theorem A}

\begin{document}

\begin{frontmatter}



\title{Weighted Trudinger-Moser inequalities \\in the subcritical Sobolev spaces and their applications 
}


\author[MI]{Masahiro Ikeda}
\ead{masahiro.ikeda@riken.jp, masahiro.ikeda@keio.jp}

\author[MS]{Megumi Sano\corref{Sano}\fnref{label1}}
\ead{smegumi@hiroshima-u.ac.jp}
\fntext[label1]{Corresponding author.}

\author[KT]{Koichi Taniguchi}
\ead{taniguchi.koichi@shizuoka.ac.jp}

\address[MI]{Faculty of Science and Technology, Keio University, 3-14-1 Hiyoshi, Kohoku-ku, Yokohama 223-8522, Japan / Center for Advanced Intelligence Project RIKEN, Japan}
\address[MS]{Laboratory of Mathematics, School of Engineering,
Hiroshima University, Higashi-Hiroshima, 739-8527, Japan}
\address[KT]{Department of Mathematical and Systems Engineering, 
Faculty of Engineering, 
Shizuoka University, 3-5-1 Johoku, Chuo-ku, Hamamatsu, Shizuoka, 432-8561, Japan}


\date{\today}

\begin{abstract}
We study boundedness, optimality and attainability of Trudinger-Moser type maximization problems in the radial and the subcritical homogeneous Sobolev spaces $\dot{W}^{1,p}_{0, \text{rad}}(B_R^N)\,(p<N)$. 
Our results give a revision of an error in \cite[Theorem C]{HL}. 
Also, our inequality converges to the original Trudinger-Moser inequality as $p \nearrow N$ including optimal exponent and concentration limit. 
Finally, we consider an application of our inequality to elliptic problems with exponential nonlinearity.

\end{abstract}

\end{frontmatter}



%
%
\section{Introduction}\label{S Intro}

Let $1< p \le N, p' := \frac{p}{p-1}$ and $B_R^N$ be open ball in $\re^N$ with center $0$, with radius $R \in (0, \infty)$. For unified notation, we set $B_{\infty}^N= \re^N$. The boundedness, the optimality and the (non-)existence of a maximizer of the following type maximization problems $T_p, T_p^{\rm rad}$ have been studied so far.
\begin{align*}
T_p &:= \sup \left\{  \int_{B_R^N} f(u) V(|x|) \,dx \,\,\middle| \,\,  u \in \dot{W}_0^{1,p} (B_R^N), \,\,\| \nabla u\|_p \le 1  \right\}\\
&\ge \sup \left\{  \int_{B_R^N} f(u) V(|x|) \,dx \,\,\middle| \,\,  u \in \dot{W}_{0, {\rm rad}}^{1,p} (B_R^N), \,\,\| \nabla u\|_p \le 1  \right\} =: T_{p}^{\rm rad}
\end{align*}

\noindent
At first, we give the following figures to explain known results about the boundedness of $T_p, T_p^{\rm rad}$ and our motivations of the present paper.

\begin{table}[h]
\caption{Boundedness of $T_p \,(p<N, R \in (0, \infty])$}
\label{subcritical table non-rad}
\centering
\begin{tabular}{|r|r|r|r|}
\hline
$f(u)$ & $V(|x|)$ & Name of ineq. & Supplement \\  \hline 
$|u|^p$ & $|x|^{-p}$ & Hardy &   \\ \hline
$|u|^q\,(p< q < p^*)$ & $|x|^{-A_q}$ & Hardy-Sobolev  &  $A_q := \frac{p^* -q}{p^* -p}p >0$ \\ \hline
$|u|^{p^*}$ & $1$ & Sobolev &  $p^*:= \frac{Np}{N-p}$ \hspace{2em} {}\\ \hline
$|u|^q\,(p^*< q < \infty)$ & 0 &  &   \\ \hline
\end{tabular}
\end{table}

\begin{table}[h]
\caption{Boundedness of $T_p^{\rm rad} \,(p<N, R \in (0, \infty])$}
\label{subcritical table}
\centering
\begin{tabular}{|r|r|r|r|}
\hline
$f(u)$ & $V(|x|)$ & Name of ineq. & Supplement \\  \hline 
$|u|^p$ & $|x|^{-p}$ & Hardy &   \\ \hline
$|u|^q\,(p< q < p^*)$ & $|x|^{-A_q}$ & Hardy-Sobolev  &  $A_q := \frac{p^* -q}{p^* -p}p >0$ \\ \hline
$|u|^{p^*}$ & $1$ & Sobolev &  $p^*:= \frac{Np}{N-p}$ \hspace{2em} {}\\ \hline
$|u|^q\,(p^*< q < \infty)$ & $|x|^{\,B_q}$ & (H\'enon or Ni) &  $B_q := \frac{q-p^*}{p^* -p}p >0$ \\ \hline
\end{tabular}
\end{table}

\begin{table}[h]
\caption{Boundedness of $T_N$ and $T_N^{\rm rad} \,(R \in (0, \infty))$}
\label{critical table}
\centering
\begin{tabular}{|r|r|r|r|}
\hline
$f(u)$ & $V(|x|)$ & Name of ineq.  & Supplement \\  \hline 
$|u|^N$ & $|x|^{-N} \( \log \frac{aR}{|x|} \)^{-N}$ & Critical Hardy  & $a \ge 1$  \\ \hline
$|u|^q\,(N< q < \infty)$ & $|x|^{-N} \( \log \frac{aR}{|x|} \)^{-\beta_q}$ & (Generalized C.H.) &  $\beta_q := \frac{N-1}{N}q +1$ \\ \hline
$\exp ( \gamma_\beta |u|^{N'} )$ & $|x|^{-\beta}\,(\beta \in (0, N))$ & Singular T.-M. & $\gamma_\beta := \alpha_N (1-\beta/N)$ \\ \hline
$\exp ( \alpha_N |u|^{N'} )$ & $1$ & Trudinger-Moser & $\alpha_N := N \w_N^{\frac{1}{N-1}}$ \\ \hline
\end{tabular}
\end{table}

\noindent
In the subcritical case $p<N$, we point out that for any $q \in (p^*, \infty)$
\begin{align*}
T_p^{\rm rad} &= \sup \left\{  \int_{B_R^N} |u|^q |x|^{B_q}\,dx \,\,\middle| \,\,  u \in \dot{W}_{0, {\rm rad}}^{1,p} (B_R^N), \,\,\| \nabla u\|_p \le 1  \right\}<\infty,\\
T_p &=\sup \left\{  \int_{B_R^N} |u|^q |x|^{B_q}\,dx \,\,\middle| \,\,  u \in \dot{W}_{0}^{1,p} (B_R^N), \,\,\| \nabla u\|_p \le 1  \right\}   = \infty,\\
T_p^{\rm rad} &= \sup \left\{  \int_{B_R^N} |u|^q \,dx \,\,\middle| \,\,  u \in \dot{W}_{0, {\rm rad}}^{1,p} (B_R^N), \,\,\| \nabla u\|_p \le 1  \right\}=\infty.
\end{align*}
Therefore, we see that the stronger growth $f(u)=|u|^q \,(q >p^*)$ than $|u|^{p^*}$ is admitted thanks to the vanishing weight function $V(|x|)=|x|^{B_q}\,(B_q >0)$ and the restriction of $\dot{W}_0^{1,p} (B_R^N)$ to $\dot{W}_{0, {\rm rad}}^{1,p} (B_R^N)$. Based on this fact, Ni \cite{N} showed the existence of a radial weak solution to the H\'enon equation:  
\begin{align*}
-\lap_p u = |x|^B |u|^{q-2} u  \,\,\text{in}\,\, B_R^N, \quad u|_{\pd B_R^N} =0\quad(B >0)
\end{align*}
for the stronger nonlinearity $|u|^{q-2}u \,(q >p^*)$. 
Inspired by Ni's result, we consider the next growth, that is, the exponential growth $f(u) = \exp (\alpha |u|^{p'})$ beyond polynomial growth $|u|^q$ in the subcritical case $p<N$. This is an analogue of the critical case $p=N$, see Table \ref{critical table}.

We can easily observe that the boundedness of $T_p^{\rm rad}$ is determined by trade-off between the growth of $f(u)$ at $+\infty$ and the vanishing speed (or singularity) of $V(|x|)$ at $0$. Namely, if we choose stronger $f(u)$, then we have to choose more rapidly vanishing (or weaker) $V(|x|)$ to obtain the boundedness of $T_p^{\rm rad}$. Based on this viewpoint, we introduce more rapidly vanishing weight function $V_p(|x|)$ than $|x|^{B_q}$ as follows. 

\begin{definition}\label{Def V_p}
Let $1< p< N$ and $R \in (0, \infty]$. Then for $x \in B_R^N \setminus \{ 0\}$ we define
\begin{align*}
V_p(|x|) &:= \( \frac{\w_p}{\w_N} \)^{p'} |x|^{-(N-1) p'} \exp \left[ -\frac{p-1}{N-p} p \( \frac{\w_p}{\w_N} \)^{\frac{1}{p-1}} \( |x|^{-\frac{N-p}{p-1}} -R^{-\frac{N-p}{p-1}} \) \right],
\end{align*}
where $\dis{\w_p = \frac{2\pi^{\frac{p}{2}}}{\Gamma \( \frac{p}{2} \)}}$ and $(\infty)^{-\frac{N-p}{p-1}} := 0$. Furthermore, we define 
\begin{align*}
F_{p, \alpha} (u) &:= \int_{B_R^N} \exp (\alpha |u|^{p'}) V_p (|x|) \,dx\quad (u \in \dot{W}_{0, {\rm rad}}^{1,p} (B_R^N)),\\
T^{\rm rad}_{p, \alpha} &:= \sup \left\{  F_{p, \alpha}(u) \,\,\middle| \,\,  u \in \dot{W}_{0, {\rm rad}}^{1,p} (B_R^N), \,\,\| \nabla u\|_p \le 1 \right\}.
\end{align*}
\end{definition}

For the weight function $V_p(|x|)$, we obtain the optimality of the growth of $F_{p, \alpha} (u)$ and the boundedness of $T_{p, \alpha}^{\rm rad}$ as follows.

\begin{theorem}\label{Thm WTM bdd}
Let $1<p< N$ and $R \in (0, \infty]$. Then
\begin{align*}
&{\rm (I)} \,\int_{B_R^N} \exp ( \alpha |u|^{\gamma}) V_p (|x|) \,dx < \infty\,\,\text{for any}\,u \in \dot{W}_{0, {\rm rad}}^{1,p} (B_R^N)\,\,\text{and}\,\,\alpha >0 \iff \gamma \le p',\\
&{\rm (II)} \,\, T^{\rm rad}_{p, \alpha} < \infty \iff \alpha \le \alpha_p:= p \w_p^{\frac{1}{p-1}}.
\end{align*}
\end{theorem}

Also, we obtain the existence of a maximizer of $T_{p, \alpha}^{\rm rad}$ as follows. 
Concerning the existence and the non-existence of the maximization problems $T_p, T_p^{\rm rad}$ for other inequalities in Table \ref{subcritical table non-rad}, Table \ref{subcritical table} and Table \ref{critical table}, see e.g. \cite{Au, Ta, Chou-Chu, CC, AS, AS TM, D, HK, II, S(JDE)}.

\begin{theorem}\label{Thm WTM max}
Let $1<p< N$ and $R \in (0, \infty]$. For any $\alpha \le \alpha_p$, there exists a maximizer of $T^{\rm rad}_{p, \alpha}$. 
\end{theorem}

In the cases where $R=1$, Theorem \ref{Thm WTM max} was already obtained by \cite[Theorem C]{HL} essentially. However, there is an  error in \cite[Theorem C]{HL} due to a miscalculation. Therefore, one of our motivations is to revise the error. 

\begin{remark}\label{Rem approximation}
Let $R \in (0, \infty)$. Since
\begin{align*}
 \frac{|x|^{-\ep} -R^{-\ep}}{\ep} = \log \frac{R}{|x|} + o(1) \,\,(\ep \to 0),
\end{align*}
we have $V_p(|x|) \to R^{-N}$ as $p \nearrow N$ for any $x \in B_R^N \setminus \{ 0\}$ which implies that
\begin{align*}
F_{p, \alpha} (u) \to \frac{1}{R^N} \int_{B_R^N} \exp (\alpha |u|^{N'}) \,dx\quad \text{as}\,\, p \nearrow N
\end{align*}
for any $u \in \dot{W}_{0, {\rm rad}}^{1,N} (B_R^N)$ and for any $\alpha >0$. 
Namely, our functional $F_{p, \alpha}(u)$ is a $W^{1,p}$-approximation of the original Trudinger-Moser functional. For another kind of $W^{1,p}$-approximation, see \cite{IH}. 
Furthermore, our optimal exponent $\alpha_p$ and the concentration level of $T_p^{\rm rad}$ are also a $W^{1,p}$-approximation of them in the critical case $p=N$, see also Remark \ref{Rem concentration limit} in \S \ref{S max}.
\end{remark}

We give a generalization of the maximization problem $T_{p,\alpha}^{\rm rad}$ as a corollary of Theorem \ref{Thm WTM bdd} and Theorem \ref{Thm WTM max}, which is corresponding to the singular Trudinger-Moser inequality in Table \ref{critical table}. 

\begin{cor}\label{Cor gene}
Let $1<p<N, R \in (0, \infty]$ and $\beta \in [0, p)$. Then  
\begin{align*}
    T^{\rm rad}_{p, \alpha, \beta} &:= \sup \left\{ \int_{B_R^N} \exp \( \alpha |u|^{p'} \) V_{p, \beta} (|x|) \,dx \,\,\middle| \,\, u \in 
 \dot{W}_{0, {\rm rad}}^{1,p} (B_R^N),\,\|\nabla u \|_{L^p (B_R^N)} \le 1 \right\} < \infty \\ 
 &\iff \alpha \le \alpha_{p, \beta} := (p-\beta) \,\w_p^{\frac{1}{p-1}}= \alpha_p \( 1- \frac{\beta}{p} \),
\end{align*}
where for $x \in B_R^N \setminus \{ 0\}$
\begin{align*}
V_{p, \beta}(|x|) := \( \frac{\w_p}{\w_N} \)^{p'} |x|^{-(N-1) p'} \exp \left[ -\frac{p-1}{N-p} (p-\beta ) \( \frac{\w_p}{\w_N} \)^{\frac{1}{p-1}} \( |x|^{-\frac{N-p}{p-1}} -R^{-\frac{N-p}{p-1}} \) \right].
\end{align*}
Furthermore, the maximization problem $T^{\rm rad}_{p, \alpha, \beta}$ is attained for any $\alpha \le \alpha_{p, \beta}$. 
\end{cor}

In the case $\beta = 0$, Corollary \ref{Cor gene} coincides with Theorem \ref{Thm WTM bdd} and Theorem \ref{Thm WTM max} since $T^{\rm rad}_{p, \alpha, 0}= T^{\rm rad}_{p, \alpha}$ and $V_{p, 0}(|x|)=V_p(|x|)$. 
As a consequence, we obtain the improved figure (Table \ref{subcritical table improved}) including Theorem \ref{Thm WTM bdd} and Corollary \ref{Cor gene}.

\begin{table}[h]
\caption{Boundedness of $T_p^{\rm rad} \,(p<N, R \in (0, \infty])$}
\label{subcritical table improved}
\centering
\begin{tabular}{|r|r|r|r|}
\hline
$f(u)$ & $V(|x|)$ & Name of ineq. & Supplement \\  \hline 
$|u|^p$ & $|x|^{-p}$ & Hardy &   \\ \hline
$|u|^q\,(p< q < p^*)$ & $|x|^{-A_q}$ & Hardy-Sobolev  &  $A_q := \frac{p^* -q}{p^* -p}p >0$ \\ \hline
$|u|^{p^*}$ & $1$ & Sobolev &  $p^*:= \frac{Np}{N-p}$ \hspace{2em} {}\\ \hline
$|u|^q\,(p^*< q < \infty)$ & $|x|^{\,B_q}$ & (H\'enon or Ni) &  $B_q := \frac{q-p^*}{p^* -p}p >0$ \\ \hline
$\exp ( \alpha_{p,\beta} |u|^{p'} )$ & $V_{p, \beta}  (|x|)\,(\beta \in (0, p))$ & Corollary \ref{Cor gene}  & $\alpha_{p, \beta} := \alpha_p (1-\beta/p)$ \\ \hline
$\exp ( \alpha_p |u|^{p'} )$ & $V_p(|x|)$ & Theorem \ref{Thm WTM bdd} & $\alpha_p := p \w_p^{\frac{1}{p-1}}$ \\ \hline
\end{tabular}
\end{table}

\noindent
{\bf Outline of the paper.} 
The paper is organized as follows. In \S \ref{S bdd}, we prove Theorem \ref{Thm WTM bdd} which is the boundedness of $T_p^{\rm rad}$ and the optimal exponent $\alpha_p$. We reduce $T_p^{\rm rad}$ to the one dimensional maximization problem $M_p$. Also, we mention that our weight function $V_p(|x|)$ is optimal in some sense, see Remark \ref{Rem V_p optimal}. In \S \ref{S max}, we prove Theorem \ref{Thm WTM max} via $M_p$. Note that the existence of a maximizer of $M_p$ is already  shown by \cite{CC, HL}. In \S \ref{S relation}, we show an equivalence between our inequality with $p \in \N$ and the original Trudinger-Moser inequality via the harmonic transplantation by \cite{ST, ST(HT)}. Also, we prove Corollary \ref{Cor gene} via the transformation based on a harmonic transplantation. Furthermore, via the transformation, we also show an equivalence between the two elliptic equations (\ref{eq v}), (\ref{eq u}) associated with the maximization problems on radial Sobolev spaces. Note that this equivalence is available only for natural numbers $p \in (1, N)$. Therefore, in \S \ref{S elliptic}, we show the existence of a radial weak solution of the elliptic equation (\ref{eq u}) for real numbers $p \in (1,N)$ via variational method without the transformation. 

\noindent
{\bf Notation.}
Set $\w_p = \frac{2\pi^{\frac{p}{2}}}{\Gamma \( \frac{p}{2} \)}$. Note that, if $p =N \in \N$, then $\w_{N}$ is surface area of unit sphere $\mathbb{S}^{N-1}$ in $\re^N$. $B_R^N$ denotes open ball in $\re^N$ with center $0$, with radius $R \in (0, \infty)$. For unified notation, we set $B_{\infty}^N= \re^N$. 
$\dot{W}_0^{1,p} (B_R^N)$ is the completion of $C_c^\infty (B_R^N)$ with respect to $\| \nabla (\cdot)\|_p$. When $R=\infty$ i.e. $B_R^N=\mathbb{R}^N$, $\|\nabla(\cdot)\|_p$ becomes seminorm yielding the same value for two functions that differ only by an additive constant. Thus, the quotient space $\dot{W}^{1,p}(\mathbb{R}^N)/\mathbb{R}$ defines a separable and reflexive Banach space since it can be identified with a closed subspace of $L^{p^*}(\mathbb{R}^N)$. For simplicity we write $\dot{W}^{1,p}(\mathbb{R}^N)$ insted of $\dot{W}^{1,p}(\mathbb{R}^N)/\mathbb{R}$ having in mind that the elements of $\dot{W}^{1,p}(\mathbb{R}^N)$ are equivalent classes. For the detail, see e.g. \cite{BGV}. 
Throughout this paper, if $u$ is a radial function that should be written as $u(x) = \tilde{u}(|x|)$ by some function $\tilde{u} = \tilde{u}(r)$, we write $u(x)= u(|x|)$ with admitting some ambiguity. Set 
$X_{\rm rad} = \{ u \in X \,\,|\,\, u(x) = u(|x|) \,\,\}$. 
We simply write the space-dependent function $u(x)$ as $u$ depending on the circumstances. Also, we use $C$ or $C_i \,(i \in \N)$ as positive constants. If necessary, we denote those by $C(\ep)$ when constants depend on $\ep$. For $q\in [1,\infty]$, we denote by $q'$ the H\"older conjugate of $q$, i.e. $1/q+1/q'=1$. 

%
%
\section{Boundedness of $T_{p, \alpha}^{\rm rad}$ and optimality: Proof of Theorem \ref{Thm WTM bdd}}\label{S bdd}

\begin{proof}{\it (Theorem \ref{Thm WTM bdd})}
\noindent
{\rm (I)} First, let $\gamma > p'$. Then we shall show that 
\begin{align*}
\int_{B_R^N} \exp ( |u|^{\gamma}) V_p (|x|) \,dx = \infty\,\,\text{for some}\,\, u \in \dot{W}_{0, {\rm rad}}^{1,p} (B_R^N).
\end{align*}
Let $\beta = \beta (\gamma) < \frac{N-p}{p}$ satisfy $\beta \gamma > \frac{N-p}{p-1}$. Also, let $\delta \in (0, \frac{R}{2})$ satisfy 
\begin{align*}
r^{-\beta \gamma} - \frac{p-1}{N-p} p \( \frac{\w_p}{\w_N} \)^{\frac{1}{p-1}} \( r^{-\frac{N-p}{p-1}} -R^{-\frac{N-p}{p-1}} \) \ge 1 \quad \text{for any} \,\, r \in (0, \delta).
\end{align*}
Consider
\begin{align*}
\varphi_\beta (x) = 
\begin{cases}
|x|^{-\beta} \quad &\text{if}\,\, x \in B_\delta^N,\\
\delta^{-1-\beta} \( 2\delta - |x|\)&\text{if}\,\, x \in B_{2\delta}^N \setminus B_\delta^N,\\
0 &\text{if} x \in B_R^N \setminus B_{2\delta}^N.
\end{cases}
\end{align*}
We easily see that $\varphi_\beta \in \dot{W}_{0, {\rm rad}}^{1,p} (B_R^N)$. On the other hand, we have
\begin{align*}
&\int_{B_R^N} \exp ( |\varphi_\beta|^{\gamma}) V_p (|x|) \,dx\\
&\ge C \int_0^\delta \exp \left[r^{-\beta \gamma} -\frac{p-1}{N-p} p \( \frac{\w_p}{\w_N} \)^{\frac{1}{p-1}} \( r^{-\frac{N-p}{p-1}} -R^{-\frac{N-p}{p-1}} \) \right] r^{-(N-1)p' + N-1}\,dr\\
&\ge C \int_0^\delta r^{-\frac{N-1}{p-1}}\,dr = \infty.
\end{align*}
Next, let $\gamma = p'$. Then we shall show that 
\begin{align*}
\int_{B_R^N} \exp ( \alpha |u|^{p'}) V_p (|x|) \,dx < \infty\,\,\text{for any}\,\, u \in \dot{W}_{0, {\rm rad}}^{1,p} (B_R^N)\,\,\text{and}\,\,\alpha >0.
\end{align*}
By the density of $C_{c, \text{rad}}^\infty (B_R^N)$ in $\dot{W}_{0, \text{rad}}^{1, p}(B_R^N)$, for any $u \in \dot{W}_{0, \text{rad}}^{1, p}(B_R^N)$ there exists $v \in C_{c, \text{rad}}^\infty (B_R^N)$ such that $\| \nabla (u-v) \|^{p'}_{p} \le \frac{\alpha_p}{\alpha 2^{p' -1}}$. Then we have
\begin{align*}
&\int_{B_R^N} \exp \( \alpha |u|^{p'} \) V_p(|x|)\,dx\\
&\le \int_{B_R^N}\exp \( \alpha 2^{p' -1} \( |u-v |^{p'} + |v|^{p'} \) \) V_p(|x|)\,dx\\
&\le \max_{x \in B_R^N} \left[ \exp \( \alpha 2^{p' -1} |v(x) |^{p'}\) \right] \int_{\Omega} \exp \( \alpha 2^{p' -1} |u-v |^{p'}  \)V_p(|x|)\,dx \\
&\le C(\alpha, u) \int_{\Omega}\exp \( \alpha 2^{p' -1} \| \nabla (u-v) \|^{p'}_p \frac{|u-v |^{p'}}{\| \nabla (u-v) \|^{p'}_p}  \)V_p(|x|)\,dx\\
&\le C(\alpha, u) T^{\rm rad}_{p, \alpha_p} < \infty.
\end{align*}
$T^{\rm rad}_{p, \alpha_p} < \infty$ follows from {\rm (II)}.
Therefore we obtain 
\begin{align*}
\int_{B_R^N} \exp ( \alpha |u|^{\gamma}) V_p (|x|) \,dx < \infty\,\,\text{for all}\,u \in \dot{W}_{0, {\rm rad}}^{1,p} (B_R^N)\,\,\text{and}\,\,\alpha >0 \iff \gamma \le p'.
\end{align*}

\noindent
{\rm (II)} First, we show that $T^{\rm rad}_{p, \alpha} = \infty$ for $\alpha > \alpha_p$. Set
\begin{align*}
u_k (r) &= 
\begin{cases}
k^{\frac{p-1}{p}} \w_p^{-\frac{1}{p}} \quad &(0 \le r \le r_k),\\
k^{\frac{p-1}{p}} \w_p^{-\frac{1}{p}} \frac{p-1}{N-p} \( \frac{\w_p}{\w_N} \)^{\frac{1}{p-1}} \frac{1}{k} \( r^{-\frac{N-p}{p-1}} -R^{-\frac{N-p}{p-1}} \) \quad &(r_k < r <R),
\end{cases}\\
&\quad \text{where}\,\,
r_k^{-\frac{N-p}{p-1}} -R^{-\frac{N-p}{p-1}} = k \frac{N-p}{p-1} \( \frac{\w_N}{\w_p} \)^{\frac{1}{p-1}}.
\end{align*}
Direct calculations show that 
\begin{align*}
\int_{B_R^N} |\nabla u_k |^p \,dx 
&= \w_N \int_{r_k}^R |u_k' (r) |^p r^{N-1} \,dr\\
&= \w_N \int_{r_k}^R \left| k^{-\frac{1}{p}} \w_p^{-\frac{1}{p}}  \( \frac{\w_p}{\w_N} \)^{\frac{1}{p-1}} r^{-\frac{N-1}{p-1}} \right|^p r^{N-1} \,dr\\
&= \( \frac{\w_p}{\w_N} \)^{\frac{1}{p-1}} \frac{1}{k} \int_{r_k}^R r^{-\frac{N-1}{p-1}} \,dr\\
&= \( \frac{\w_p}{\w_N} \)^{\frac{1}{p-1}} \frac{1}{k} \frac{p-1}{N-p} \( r_k^{-\frac{N-p}{p-1}} -R^{-\frac{N-p}{p-1}} \) =1
\end{align*}
and
\begin{align*}
&\int_{B_R^N} \exp (\alpha |u_k|^{p'}) V_p (|x|) \,dx 
\ge \exp \( \alpha k \w_p^{-\frac{1}{p-1}} \) \int_{B_{r_k}^N} V_p (|x|) \,dx\\
&= \exp \( \alpha k \w_p^{-\frac{1}{p-1}} \) \( \frac{\w_p}{\w_N} \)^{p'} \w_N \int_0^{r_k} r^{-\frac{N-1}{p-1}} \exp \left[ -\frac{p-1}{N-p} p \( \frac{\w_p}{\w_N} \)^{\frac{1}{p-1}} \( r^{-\frac{N-p}{p-1}} -R^{-\frac{N-p}{p-1}} \) \right]\,dr\\
&= \exp \( \alpha k \w_p^{-\frac{1}{p-1}} \) \frac{\w_p}{p}  \exp \left\{ -\frac{p-1}{N-p} p \( \frac{\w_p}{\w_N} \)^{\frac{1}{p-1}} \( r_k^{-\frac{N-p}{p-1}} -R^{-\frac{N-p}{p-1}} \) \,\right\}\\
&= \frac{\w_p}{p}  \exp \left\{  k \w_p^{-\frac{1}{p-1}} \( \alpha - p \w_p^{\frac{1}{p-1}} \) \right\} \to \infty \quad (k \to \infty).
\end{align*}
Therefore, $T^{\rm rad}_{p, \alpha} = \infty$ for $\alpha > \alpha_p$. 
Next, we show that $T^{\rm rad}_{p, \alpha} < \infty$ for $\alpha = \alpha_p$. Consider the following Moser type transformation (Ref. \cite{M}).

\begin{align*}
w(t) = \alpha_p^{\frac{p-1}{p}} u(r), \quad t=  \frac{p-1}{N-p} p \( \frac{\w_p}{\w_N} \)^{\frac{1}{p-1}} \( r^{-\frac{N-p}{p-1}} -R^{-\frac{N-p}{p-1}} \), \, \alpha_p = p \w_p^{\frac{1}{p-1}}
\end{align*}

\noindent
Since
\begin{align*}
\frac{dt}{dr}=  - p \( \frac{\w_p}{\w_N} \)^{\frac{1}{p-1}} r^{-\frac{N-1}{p-1}},
\end{align*}
we have
\begin{align*}
\int_{B_R^N} |\nabla u |^p \,dx 
&= \w_N \int_0^R |u'(r)|^p r^{N-1}\,dr\\
&= \w_N \,\alpha_p^{1-p} \int_0^\infty |w'(t)|^p \left| \frac{dt}{dr} \right|^{p-1} r^{N-1}\,dt\\
&= \int_0^\infty |w'(t)|^p \,dt \le 1,\\
\int_{B_R^N} \exp (\alpha_p |u|^{p'}) V_p (|x|) \,dx 
&= \w_N \int_0^R  \exp (\alpha_p |u(r)|^{p'}) \( \frac{\w_p}{\w_N} \)^{p'} r^{-\frac{N-1}{p-1}} e^{-t} \,dr\\
&= \w_N \( \frac{\w_p}{\w_N} \)^{p'} \int_0^\infty e^{|w(t)|^{p'} -t} r^{-\frac{N-1}{p-1}} \left| \frac{dr}{dt} \right| \,dt\\
&= \frac{\w_p}{p} \int_0^\infty e^{|w(t)|^{p'} -t} \,dt.
\end{align*}

\noindent
Here, we recall the following result.

\begin{ThmA}(\cite{Adams} Lemma 1)
Let $a(s,t)$ be a non-negative measurable function on $(-\infty, + \infty) \times [0, + \infty)$ such that (a.e.)
\begin{align*}
a(s, t) \le 1, \,\, \text{when}\,\, 0< s<t,\\
\sup_{t>0} \( \int_{-\infty}^0 + \int_t^\infty a(s, t)^{p'} \,ds \)^{1/p'} =b < \infty.
\end{align*}
Then there is a constant $c_0 = c_0(p, b)$ such that if for $\phi \ge 0$, 
\begin{align*}
\int_{-\infty}^\infty \phi(s)^p \,ds \le 1,
\end{align*}
then
\begin{align*}
\int_{0}^\infty e^{-F(t)} \,dt \le c_0,
\end{align*}
where
\begin{align*}
F(t) = t- \( \int_{-\infty}^\infty a(s, t) \phi (s) \,ds \)^{p'}.
\end{align*}
\end{ThmA}

\noindent
Now, we apply Theorem A as follows.
\begin{align*}
\phi (s) = 
\begin{cases}
w'(s)\quad &(s>0),\\
0 \quad &(s \le 0),
\end{cases}
\quad a(s,t) = 
\begin{cases}
1 \quad (0< s < t),\\
0 \quad (s \le 0, \,s \ge t)
\end{cases}
\end{align*} 

\noindent
Then, we obtain the following one dimensional maximization problem $M_p$ from $T^{\rm rad}_{p, \alpha_p}$, see also \cite{J}. 
\begin{align}\label{1dim problem}
T^{\rm rad}_{p, \alpha_p} &= \sup \left\{ \int_0^\infty e^{ |w|^{p'} -t} \,dt \, \middle| \, w \in C^1 [0, \infty ), w(0) =0, w' (t) \ge 0, \int_0^\infty |w'|^p \,dt \le 1\right\} \notag \\
&=: M_p < \infty 
\end{align}

\noindent
Therefore, $T^{\rm rad}_{p, \alpha} < \infty$ for $\alpha = \alpha_p$. 
Since $F_{p, \alpha}(u) \le F_{p, \alpha_p} (u)$ for $\alpha \le \alpha_p$, we have $T^{\rm rad}_{p, \alpha} < \infty \iff \alpha \le \alpha_p$. 
\end{proof}

\begin{remark}\label{Rem V_p optimal}
 Even if we replace $|x|^{-(N-1)p'}$ with $|x|^{-(N-1)p' -\ep} \,(\ep >0)$ in $V_p(|x|)$, we can show that $\alpha_p$ is optimal, namely, $T_{p, \alpha}^{\rm rad} < \infty$ for $\alpha < \alpha_p$ and $T_{p, \alpha}^{\rm rad} = \infty$ for $\alpha > \alpha_p$. However, in this case, we can show $T_{p, \alpha_p}^{\rm rad} = \infty$ by using the same test function $u_k$ in {\rm (II)} in the proof of Theorem \ref{Thm WTM bdd}. This is not an analogue of the result for the original TM inequality $(p=N)$ by Moser \cite{M}. In this sense, $|x|^{-(N-1)p'}$ in $V_p(|x|)$ is an optimal singularity to obtain the same result as \cite{M}.
\end{remark}

%
%
\section{Existence of a maximizer of $T_{p, \alpha}^{\rm rad}$: Proof of Theorem \ref{Thm WTM max}}\label{S max}

To show Theorem \ref{Thm WTM max}, we need the following lemma.

\begin{lemma}\label{Lem cpt}
Let $1<p<N, V_p(|x|)$ be given in Definition \ref{Def V_p}, and $a(x)\ge 0$ be a bounded function on $B_R^N$. Then for any $q \in [1, \infty)$ the embedding $\dot{W}_{0, {\rm rad}}^{1,p} (B_R^N) \hookrightarrow L^q(B^N_R; a(x) V_p(|x|)\,dx)$ is compact. Moreover the following estimate holds.
\begin{align*}
    \left(\int_{B^N_R}|u(x)|^qa(x)V_p(|x|)dx\right)^{1/q}
    \le\|a\|_{\infty}^{\frac{1}{q}}\omega_p^{\frac{1}{q}-\frac{1}{p}}p^{-1+\frac{1}{p}-\frac{1}{q}}\Gamma\left(\left(1-\frac{1}{p}\right)q+1\right)^{1/q} \|\nabla u\|_p.
\end{align*}
\end{lemma}

\begin{proof}{\it (Lemma \ref{Lem cpt})} First, we consider the case where $R<\infty$. Let 
\begin{align*}
    u_m \rightharpoonup u \,\,{\rm in} \,\, \dot{W}_{0, {\rm rad}}^{1,p}(B_R^N).
\end{align*}
Note that the embedding $\dot{W}_{0, {\rm rad}}^{1,p}(B_R^N \setminus \overline{B_\delta^N}) \hookrightarrow L^q (B_R^N \setminus\overline{ B_\delta^N} )$ is compact for any $\delta, R >0$ because of the boundedness of the Jacobian $r^{N-1}$ in $(\delta, R)$ and the compactness of the one dimensional Sobolev space $\dot{W}^{1,p}_0(\delta, R) \hookrightarrow L^q(\delta , R)$. By the Radial Lemma:
\begin{align}\label{Radial Lemma}
    | \,u(|x|)\, | \le \( \frac{p-1}{N-p} \)^{\frac{p-1}{p}} \w_{N}^{-\frac{1}{p}}  \| \nabla u\|_{L^p(B^N_R  \setminus B_{|x|}^N)} \left| \, |x|^{-\frac{N-p}{p-1}} -  R^{-\frac{N-p}{p-1}} \,\right|^{\frac{p-1}{p}}\quad ({\rm a.e.} \,\, x \in B_R^N)
\end{align}
we have
\begin{align*}
     &\int_{B_R^N} |u_m -u|^q  a(x) V_p (|x|) \,dx \\
     &\le  C \| \nabla (u_m -u) \|_p^q \int_{B_\delta^N} |x|^{-\frac{N-p}{p} q} V_p(|x|) \,dx + \( \max_{x \in B_R^N \setminus B_\delta^N} a(x) V_p (|x|) \) \int_{B_R^N \setminus B_\delta^N} |u_m-u|^q \,dx\\
     &=: D_1(\delta) + D_2 (m, \delta).
\end{align*}
Since for fixed $\delta>0$ the constant $D_2(m, \delta) \to 0$ as $m \to 0$ and $D_1(\delta) \to 0$ as $\delta \to 0$, we have $u_m \rightharpoonup u \,\,{\rm in} \,\, L^q(B^N_R; a(x) V_p(|x|)\,dx)$.
Next, we consider the case where $R=\infty$. If $q \le \frac{Np}{N-p}$, then we can estimate by using the Sobolev inequality as follows.
\begin{align*}
     &\int_{\re^N} |u_m -u|^q a(x) V_p (|x|) \,dx \\
     &\le D_1(\delta) + D_2 (m, \delta, T) + \( \max_{x \in \re^N \setminus B_T^N} a(x) V_p(|x|) \) \( \int_{\re^N} |\nabla (u_m- u)|^p \,dx \)^{\frac{q}{p}}\\
     &\le D_1(\delta) + D_2 (m, \delta, T) + D_3 (T) \to 0 \,\,(m, T \to \infty, \,\delta \to 0)
\end{align*}
If $q > \frac{Np}{N-p}$, then we can estimate by using the Radial lemma as follows.
\begin{align*}
     &\int_{\re^N} |u_m -u|^q  a(x) V_p (|x|) \,dx \\
     &\le D_1(\delta) + D_2 (m, \delta, T) + C  \| \nabla (u_m -u) \|_p^q \int_{\re^N \setminus B_T^N} |x|^{-\frac{N-p}{p} q - (N-1)p'} \,dx\\
     &\le D_1(\delta) + D_2 (m, \delta, T) + D_3 (T) \to 0 \,\,(m, T \to \infty, \,\delta \to 0)
\end{align*}
Therefore we have we have $u_m \rightharpoonup u \,\,{\rm in} \,\, L^q(B^N_R; a(x) V_p(|x|)\,dx)$ for any $R \in (0, \infty]$.

Finally we prove the estimate. By the boundedness of $a$, the radial lemma (\ref{Radial Lemma}) and changing variables with
\begin{align*}
    t=\frac{p-1}{N-p}p\left(\frac{\omega_p}{\omega_N}\right)^{\frac{1}{p-1}} \( r^{-\frac{N-p}{p-1}}- R^{-\frac{N-p}{p-1}} \)
\end{align*}
the following estimates hold.
\begin{align*}
&\int_{B^N_R}|u(x)|^qa(x)V_p(|x|)dx\\
&\le \|a\|_{\infty}\left(\frac{p-1}{N-p}\right)^{\frac{(p-1)q}{p}}\omega_N^{-\frac{q}{p}}\left(\frac{\omega_p}{\omega_N}\right)^{p'}\|\nabla u\|_p^q\\
&\,\,\times \int_{B^N_R}|x|^{-(N-1)p'}\left||x|^{-\frac{N-p}{p-1}}-R^{-\frac{N-p}{p-1}}\right|^{\frac{(p-1)q}{p}}
\exp\left[-\frac{p-1}{N-p}p\left(\frac{\omega_p}{\omega_N}\right)^{\frac{1}{p-1}}\left(|x|^{-\frac{N-p}{p-1}}-R^{-\frac{N-p}{p-1}}\right)\right]dx\\
&= \|a\|_{\infty}\left(\frac{p-1}{N-p}\right)^{\frac{q(p-1)}{p}}\omega_N^{1-\frac{q}{p}}\left(\frac{\omega_p}{\omega_N}\right)^{p'}\|\nabla u\|_p^q\\
&\,\,\times \int_{0}^{R}r^{-\frac{N-1}{p-1}}\left( r^{-\frac{N-p}{p-1}}-R^{-\frac{N-p}{p-1}}\right)^{\frac{(p-1)q}{p}}
\exp\left[-\frac{p-1}{N-p}p\left(\frac{\omega_p}{\omega_N}\right)^{\frac{1}{p-1}}\left(r^{-\frac{N-p}{p-1}}-R^{-\frac{N-p}{p-1}}\right)\right]dr\\
&= \|a\|_{\infty}\omega_p^{1-\frac{q}{p}}p^{\frac{(-p+1)q}{p}-1} \|\nabla u\|_p^q\int_0^{\infty}t^{\frac{(p-1)q}{p}}e^{-t}dt.
\end{align*}
\end{proof}


\begin{proof}{\it (Theorem \ref{Thm WTM max})}
If $\alpha < \alpha_p$, then we can show easily the existence of a maximizer by the compactness argument. In fact, let $\{ u_m \} \subset \dot{W}_{0, {\rm rad}}^{1,p} (B_R^N)$ be a maximizing sequence of $T_{p, \alpha}^{\rm rad}$. Namely,
\begin{align*}
    \int_{B_R^N} |\nabla u_m |^p \,dx \le 1, \,\, \int_{B_R^N} e^{ \alpha |u_m|^{p'} } V_p(|x|) \,dx \to T_{p, \alpha}^{\rm rad}.
\end{align*}
Since $\{ u_m \}$ is bounded in $\dot{W}_{0, {\rm rad}}^{1,p}(B_R^N)$, there exists $u_* \in \dot{W}_{0, {\rm rad}}^{1,p}(B_R^N)$ such that
\begin{align*}
    u_m \rightharpoonup u_* \,\,{\rm in} \,\, \dot{W}_{0, {\rm rad}}^{1,p}(B_R^N), \,\,
    \int_{B_R^N} |\nabla u_* |^p \,dx \le \liminf_{m \to \infty} \int_{B_R^N} |\nabla u_m |^p \,dx \le 1.
\end{align*}
Let $q>2$ satisfy $\frac{2}{q}+ \frac{\alpha}{\alpha_p} =1$. 
By Lemma \ref{Lem cpt}, we have
\begin{align*}
    u_m \to u_* \,\,{\rm in} \,\,L^q(B_R^N ; V_p(|x|) \,dx).
\end{align*}
Therefore we have
\begin{align*}
    &\left| \int_{B_R^N} \( e^{ \alpha |u_m|^{p'} } - e^{ \alpha |u_*|^{p'} }  \) V_p(|x|) \,dx \right| \\
    &\le \alpha p' \int_{B_R^N} \( |u_m|^{p'-1} e^{ \alpha |u_m|^{p'} } + |u_*|^{p'-1} e^{ \alpha |u_*|^{p'} }  \) |u_m-u_*| V_p(|x|) \,dx\\
    &\le \alpha p' \( T_{p, \alpha_p}^{\rm rad} \)^{\frac{\alpha}{\alpha_p}} C \( \int_{B_R^N} |u_m -u_*|^q V_p(|x|) \,dx \)^{\frac{1}{q}} 
    \to 0
\end{align*}
which implies that $u_*$ is a maximizer of $T_{p, \alpha}^{\rm rad}$. 


Let $\alpha =\alpha_p$. In this case, the maximization problem $T_{p,\alpha_p}^{\rm rad}$ is equivalent to the one-dimensional maximization problem $M_p$ in (\ref{1dim problem}). 
Therefore it is enough to show the existence of a maximizer of $M_p$. In the case $p \in \N$, the existence of a maximizer of $M_p$ was shown by \cite{CC} and the case $p \not\in \N$ was shown by \cite{HL}. 
\end{proof}

\begin{remark}\label{Rem concentration limit}
    In our problem $T_{p, \alpha_p}^{\rm rad}$, the concentration level is $1+ \exp \( \int_1^\infty \frac{s^{p-1} -1}{s^p (s-1)} \,ds \)$.
    This coincides with {\it the Carleson-Chang limit}: $1+ e^{1+ \frac{1}{2}+ \cdots \frac{1}{N-1}}$ in the case $p =N \in \N$, and this is a $W^{1,p}$-approximation of it, see \cite{HL}. 
\end{remark}


Also, we can obtain the existence result of the following maximization problem with subcritical growth in a similar way to  the case $\alpha < \alpha_p$ in the proof of Theorem \ref{Thm WTM max}. We omit the proof. 

\begin{prop}\label{prop sub}(Subcritical growth)
    Let $\gamma < p', \alpha >0$ and $V_{p, \beta}(|x|)$ be given in Corollary \ref{Cor gene}. Then 
    \begin{align*}
    \sup \left\{ \int_{B_R^N} e^{\alpha |u|^\gamma} V_{p, \beta} (|x|) \,dx \,\,\middle| \,\,  u \in \dot{W}_{0, {\rm rad}}^{1,p} (B_R^N), \,\,\| \nabla u\|_p \le 1 \right\}
    \end{align*}
    is attained. 
\end{prop}



As a direct application of our inequalities, we can show the existence of a weak solution of the Euler-Lagrange equation as follows. We use the notation $\lap_p u= {\rm div} (|\nabla u|^{p-2} \nabla u)$.

\begin{cor}\label{EL eq}(Euler-Lagrange equation)
    Assume that 
    \begin{align*}
        \gamma < p' \,(\text{Subcritical growth}) \quad \text{and}\quad \alpha >0
    \end{align*}
    or 
    \begin{align*}
        \gamma = p' \,(\text{Critical growth}) \quad \text{and} \quad 0< \alpha \le \alpha_{p, \beta}.
    \end{align*}
    Then 
    \begin{align*}
        (EL)\quad \begin{cases}
            &-\lap_p u =  \frac{u^{\gamma -1} e^{\alpha u^\gamma} V_{p,\beta} (|x|)}{\int_{B_R^N} u^\gamma e^{\alpha u^\gamma} V_{p, \beta} (|x|) \,dx} \quad {\rm in} \quad B_R^N,\\
            &u\ge 0 \quad {\rm in} \,\, B_R^N, \quad  u=0 \quad {\rm on} \,\, \pd B_R^N
        \end{cases}
    \end{align*}
    admits a nontrivial radial weak solution, namely, $u \in \dot{W}_{0, \rm rad}^{1,p} (B_R^N) \setminus \{ 0\}$ satisfies
\begin{align*}
    \int_{B_R^N} |\nabla u|^{p-2} \nabla u \cdot \nabla \varphi \,dx 
    =  \frac{\int_{B_R^N} u^{\gamma -1} e^{\alpha u^\gamma} V_{p,\beta} (|x|) \,\varphi \,dx }{\int_{B_R^N} u^\gamma e^{\alpha u^\gamma} V_{p, \beta} (|x|) \,dx} 
\end{align*}
    for any $\varphi \in \dot{W}_{0}^{1,p} (B_R^N)$.
\end{cor}

\begin{proof}{\it (Corollary \ref{EL eq})}
By Proposition \ref{prop sub} and Corollary \ref{Cor gene}, we have a nonnegative maximizer $u \in \dot{W}_{0, \rm rad}^{1,p} (B_R^N) \setminus \{ 0\}$ with $\| \nabla u \|_p \le 1$. Then $u$ satisfies
\begin{align}\label{weak form}
    \int_{B_R^N} |\nabla u|^{p-2} \nabla u \cdot \nabla \varphi \,dx 
    =  \lambda \int_{B_R^N} u^{\gamma -1} e^{\alpha u^\gamma} V_{p,\beta} (|x|) \,\varphi \,dx  
\end{align}
for any $\varphi \in \dot{W}_{0, \rm rad}^{1,p} (B_R^N)$ and for some $\lambda \in \re$.  
First, we see that $\| \nabla u\|_p =1$. In fact, if $\| \nabla u \|_p < 1$, we set $v=u/ \|\nabla u\|_p$ and obtain 
\begin{align*}
    T= \int_{B_R^N} e^{\alpha |u|^\gamma } V_{p, \beta} (|x|) \,dx
    = \int_{B_R^N} e^{\alpha \| \nabla u\|_p^\gamma  |v|^\gamma } V_{p, \beta} (|x|) \,dx
    < \int_{B_R^N} e^{\alpha |v|^\gamma } V_{p, \beta} (|x|) \,dx \le T.
\end{align*}
This is a contradiction. Thus, $\| \nabla u \|_p =1$ which implies $\dis{\la = \( \int_{B_R^N} u^{\gamma} e^{\alpha u^\gamma} V_{p,\beta} (|x|) \,dx \)^{-1}}$. Next, we show that (\ref{weak form}) holds true for any $\varphi \in \dot{W}_{0}^{1,p} (B_R^N)$. We use the polar coordinate $x = r\w \,(r=|x|, \w \in \mathbb{S}^{N-1})$. For $\varphi \in \dot{W}_{0}^{1,p} (B_R^N)$, consider the following radial function.
\begin{align*}
    \tilde{\varphi} (r) = \frac{1}{\w_N} \int_{\mathbb{S}^{N-1}} \varphi (r \w) \,dS_{\w} \quad (0 \le r < R)
\end{align*}
Then we have
\begin{align}\label{weak form radial}
    \int_{B_R^N} |\nabla u|^{p-2} \nabla u \cdot \nabla \tilde{\varphi} \,dx 
    =  \lambda \int_{B_R^N} u^{\gamma -1} e^{\alpha u^\gamma} V_{p,\beta} (|x|) \,\tilde{\varphi} \,dx  
\end{align}
Since $|\nabla u(x)|=|u'(r)|$ and $\dis{\nabla u(x) \cdot \nabla \varphi (x) = u'(r) \, \frac{\pd \varphi}{\pd r} (r\w)}$, we have
\begin{align*}
    ({\rm L.H.S. \,of}\, (\ref{weak form radial}))
    &= \w_N \int_0^R |u'(r)|^{p-2} u'(r) \,\tilde{\varphi}'(r) r^{N-1} \,dr \\
    &= \int_0^R |u'(r)|^{p-2} u'(r) \( \int_{\mathbb{S}^{N-1}} \frac{\pd \varphi}{\pd r} (r\w) \,dS_{\w} \) r^{N-1} \,dr\\
    &= \int_{B_R^N} |\nabla u|^{p-2} \nabla u \cdot \nabla \varphi \,dx,\\
    ({\rm R.H.S. \,of}\, (\ref{weak form radial}))
    &= \la \w_N \int_0^R u(r)^{\gamma -1} e^{\alpha u(r)^\gamma} V_{p,\beta} (r) \,\tilde{\varphi}(r) r^{N-1}\,dr \\
    &= \la  \int_0^R u(r)^{\gamma -1} e^{\alpha u(r)^\gamma} V_{p,\beta} (r)  \( \int_{\mathbb{S}^{N-1}} \varphi  (r\w) \,dS_{\w} \) r^{N-1}\,dr \\
    &=  \lambda \int_{B_R^N} u^{\gamma -1} e^{\alpha u^\gamma} V_{p,\beta} (|x|) \, \varphi \,dx.
\end{align*}
Therefore, we see that the weak form (\ref{weak form}) holds true for any $\varphi \in \dot{W}_{0}^{1,p} (B_R^N)$.
\end{proof}

We will discuss the existence of a weak solution of elliptic equations with general nonlinearity in \S \ref{S elliptic}. 

%
%
\section{Relation between the weighted TM and the original TM inequalities: Proof of Corollary \ref{Cor gene}}\label{S relation}

In Remark \ref{Rem approximation} and Remark \ref{Rem concentration limit}, we see that our inequality is a $W^{1,p}$-approximation of the original Trudinger-Moser inequality including the best exponent $\alpha_p$ and the concentration limit $1+ \exp \( \int_1^\infty \frac{s^{p-1} -1}{s^p (s-1)} \,ds \)$. 
In this section, we give an equivalence between our inequality with $p \in \N$ and the original Trudinger-Moser inequality via the harmonic transplantation by \cite{ST, ST(HT)}: 
Let $m, N \in \N$ and $1< p= N < m$. 
Consider the following harmonic transplantation for radial functions $u, v$.
\begin{align}\label{HT p in N}
    u(|x|) = v(|y|), \,\,\text{where}\,\,  \frac{p-1}{m-p} \w_m^{-\frac{1}{p-1}} \( |x|^{-\frac{m-p}{p-1}} - R^{-\frac{m-p}{p-1}} \) = \w_N^{-\frac{1}{N-1}} \log \frac{1}{|y|}
\end{align}
Then we have the equivalence between two norms of the subcritical Sobolev space $\dot{W}_{0, {\rm rad}}^{1,p} (B_R^m)$ and the critical Sobolev space $\dot{W}_{0, {\rm rad}}^{1,N} (B_1^N)$ as follows.
\begin{align*}
    \| \nabla u \|_{L^p(B_R^m)} = \| \nabla v \|_{L^N (B_1^N)}
\end{align*}
Also, we have the equivalence between our functional and the original Trudinger-Moser functional as follows.
\begin{align*}
    \int_{B_R^m} \exp (\alpha |u|^{p'}) V_p (|x|) \,dx = \int_{B_1^N} \exp \( \alpha |v|^{N'} \) \,dy 
\end{align*}
Therefore, we see that the special case $p \in \N$ of our maximization problem $T_{p,\alpha}^{\rm rad}$ in $B_R^{m}$ 
is equivalent to the original Trudinger-Moser maximization problem:
\begin{align*}
    \sup \left\{ \,\int_{B_1^N} \exp \( \alpha |v|^{N'} \) \,dy \,\,\middle| \,\, v \in 
 \dot{W}_{0, {\rm rad}}^{1,N} (B_1^N),\,\|\nabla v \|_{L^N (B_1^N)} \le 1 \, \right\}
\end{align*}
via the harmonic transplantation (\ref{HT p in N}).

\begin{proof}{\it (Corollary \ref{Cor gene})}
    Consider the following transplantation for radial functions $u, v$.
\begin{align}\label{HT p to p}
    u(|x|) = \( \frac{p}{p-\beta} \)^{\frac{p-1}{p}} v(|y|), \,\,\text{where}\,\,  \beta  \( |x|^{-\frac{N-p}{p-1}} - R^{-\frac{N-p}{p-1}} \) = p  \( |y|^{-\frac{N-p}{p-1}} - R^{-\frac{N-p}{p-1}} \)
\end{align}
Then we have
\begin{align*}
     \| \nabla u \|_{L^p(B_R^N)} &= \| \nabla v \|_{L^p (B_R^N)},\\
      \int_{B_R^N} \exp (\alpha |u|^{p'}) V_{p, \beta} (|x|) \,dx &= \int_{B_R^N} \exp \( \alpha \frac{p}{p-\beta} |v|^{p'} \) V_p (|y|) \,dy.
\end{align*}
Therefore, we obtain Corollary \ref{Cor gene} from Theorem \ref{Thm WTM bdd} and Theorem \ref{Thm WTM max}. 
\end{proof}

\begin{remark}
    If we consider the composed transformation by two transformations (\ref{HT p in N}), (\ref{HT p to p}), we see that the generalized maximization problem $T_{p,\alpha, \beta}^{\rm rad}$ is also equivalent to the original Trudinger-Moser maximization problem. 
\end{remark}

 In the special case $p \in \N$, we see that Table \ref{subcritical table improved} $(p=N<m, m:$ dimension) is corresponding to Table \ref{critical table} $(p=N, N:$ dimension) each other via the harmonic transplantation (\ref{HT p in N}). Furthermore, in this case, we see that radial weak solutions of the following two elliptic equations are also equivalent each other via (\ref{HT p in N}).

 \begin{align}\label{eq v}
     -\lap_N v &= f(v) \,\,\text{in}\,\, B_1^N, \quad v|_{\pd B_1^N} =0\\
     \label{eq u}
     -\lap_p u &= f(u) V_p(|x|) \,\,\text{in}\,\, B_R^m, \quad v|_{\pd B_R^m} =0
 \end{align}

 \begin{prop}\label{Prop equivalence}
 Let $1<p=N < m$. 
     If $v \in \dot{W}_{0, \rm rad}^{1,N}(B_1^N)$ is a radial weak solution of (\ref{eq v}), then $u \in \dot{W}_{0, \rm rad}^{1,p}(B_R^m)$ which is given by (\ref{HT p in N}) is the radial weak solution of (\ref{eq u}), and vice versa. 
 \end{prop}

 \begin{proof}{\it ( Proposition \ref{Prop equivalence})}
For any $\varphi \in \dot{W}_{0, \rm rad}^{1,N}(B_1^N)$, $v$ satisfies
\begin{align}\label{weak form v}
    \int_{B_1^N} |\nabla v |^{N-2} \nabla v \cdot \nabla \varphi \,dy = \int_{B_1^N} f(v) \varphi \,dy.
\end{align}
Set $|x|=r, |y|=s$ for $x \in B_R^m$ and $y \in B_1^N$, and 
\begin{align*}
    u(r) = v(s), \,\,\phi(r)=\varphi(s), \,\,\text{where}\,\,  \frac{p-1}{m-p} \w_m^{-\frac{1}{p-1}} \( r^{-\frac{m-p}{p-1}} - R^{-\frac{m-p}{p-1}} \) = \w_N^{-\frac{1}{N-1}} \log \frac{1}{s}.
\end{align*}
Since 
\begin{align*}
    s \frac{dr}{ds} = r^{\frac{m-1}{p-1}} \( \frac{\w_m}{\w_N} \)^{\frac{1}{p-1}}\,\,\text{and}\,\, \( \frac{ds}{dr} \)^p = s^p r^{-(m-1)p'}  \( \frac{\w_N}{\w_m} \)^{p'} = V_p(r),
\end{align*}
we have
\begin{align*}
    (\text{L.H.S. of (\ref{weak form v})}) 
    &= \w_N \int_0^1 |v'(s)|^{N-2} v'(s) \varphi'(s) s^{N-1}\,ds  \\
    &= \w_N \int_0^R |u'(r)|^{p-2} u'(r) \phi'(r) \(s \frac{dr}{ds} \)^{p-1}\,dr\\
    &= \w_m \int_0^R |u'(r)|^{p-2} u'(r) \phi'(r) r^{m-1}\,dr
    = \int_{B_R^m} |\nabla u |^{p-2} \nabla u \cdot \nabla \phi \,dx,\\
     (\text{R.H.S. of (\ref{weak form v})}) 
    &= \w_N \int_0^1 f(v(s)) \varphi(s) s^{N-1}\,ds  \\
    &= \w_N \int_0^R f(u(r)) \phi(r) \(s \frac{dr}{ds} \)^{p-1} \( \frac{ds}{dr} \)^p \,dr\\
    &= \w_m \int_0^R f(u(r)) \phi(r) V_p(r) r^{m-1}\,dr
    = \int_{B_R^m}  f(u) V_p(|x|) \phi \,dx.
\end{align*}
 \end{proof}


Also, we can check that a radial classical solution of (\ref{eq v}) in $B_1^N \setminus \{ 0\}$ is equivalent to the radial classical solution of (\ref{eq u}) in $B_R^m \setminus \{ 0\}$. In fact, since
\begin{align*}
    \lap_{p} u = (p-1) \frac{|u'(r)|^{p-2}}{r} \( u''(r) r + \frac{m-1}{p-1} u'(r) \),
\end{align*}
we have
\begin{align*}
    \lap_{N} v &= (N-1) \frac{|v'(s)|^{N-2}}{s} \frac{d}{ds} \( v'(s) s \) \\
    &= (p-1) \frac{|u'(r)|^{p-2}}{s} \(\frac{dr}{ds} \)^{p-1} \frac{d}{dr} \(u'(r) r^{\frac{m-1}{p-1}} \(\frac{\w_m}{\w_N}\)^{\frac{1}{p-1}} \)\\
    &=  \(\frac{\w_m}{\w_N}\)^{\frac{1}{p-1}} (p-1)  \frac{|u'(r)|^{p-2}}{s^p}  r^{m-1} \frac{\w_m}{\w_N} \(u''(r) r^{\frac{m-1}{p-1}} + \frac{m-1}{p-1} u'(r) r^{\frac{m-1}{p-1} -1} \)\\
    &=  \(\frac{\w_m}{\w_N}\)^{p'} \frac{r^{(m-1)p'}}{s^p} (p-1) \frac{|u'(r)|^{p-2}}{r} \( u''(r) r + \frac{m-1}{p-1} u'(r) \)
    = V_p(r)^{-1} \lap_p u.
\end{align*}
Therefore, if $v \in C_{\rm rad}^2 (B_1^N \setminus \{ 0\})$ satisfies (\ref{eq v}) in $B_1^N \setminus \{ 0\}$ in the classical sense, then the transplanted function $u \in C_{\rm rad}^2 (B_R^m \setminus \{ 0\})$ by (\ref{HT p in N}) satisfies (\ref{eq u}) in $B_R^m \setminus \{ 0\}$ in the classical sense, and vice versa.

By using Proposition \ref{Prop equivalence}, we observe that the existence of a radial weak solution of (\ref{eq u}) follows from known results for (\ref{eq v}) in the case $p \in \N$ (Ref. \cite{A,FMR,SS} etc.). 
However, for real numbers $p$, we do not have any equivalence such as Proposition \ref{Prop equivalence}. 
Therefore, in \S \ref{S elliptic}, we study directly the existence of a radial weak solution of (\ref{eq u}) for real numbers $p$ via variational method without transformation. 

%
%
\section{Application of the weighted TM inequality to the elliptic equation}\label{S elliptic}
In this section, we apply the weighted TM inequality (Theorem \ref{Thm WTM bdd}) to the elliptic equation:
\begin{equation}\label{Elliptic eq}
   -\Delta_p u =  V(|x|) f(u) \,\,\text{in}\,\, B_R^N,\quad u|_{\pd B_R^N} =0,
\end{equation}
where $1< p < N, R\in (0,\infty]$ and $\lap_p u = {\rm div} (|\nabla u|^{p-2} \nabla u)$. 
Based on Theorem \ref{Thm WTM bdd}, we say that $f$ has {\it subcritical growth} at $\infty$ if for any $\alpha >0$
\begin{align}\label{subcritical cond}
    \lim_{|t| \to  \infty} |f(t)| \,e^{-\alpha |t|^{p'}} =0
\end{align}
and $f$ has {\it critical growth} at $\infty$ if there exists $\alpha_0 >0$ such that 
\begin{align}\label{critical cond}
    \lim_{|t| \to  \infty} |f(t)| \,e^{-\alpha |t|^{p'}} =0\, \,(\forall \alpha > \alpha_0), \quad 
    \lim_{|t| \to  \infty} |f(t)| \,e^{-\alpha |t|^{p'}} =\infty\,\,(\forall \alpha < \alpha_0).
\end{align}
We consider nonlinearities which have subcritical or critical growth in the above sense. 
In addition, we introduce the following assumptions.
\begin{align*}
    (A1) \quad &f: \re \to \re\,\,\text{is continuous},\, f(-t) \le f(0) = 0 \le f(t) \,\,\text{for any}\,\,t >0.\\
    (A2) \quad &V \not\equiv 0, V \ge 0 \,\, \text{a.e. in}\,\, B_R^N, \,\text{and there exists} \,\, C_0 >0\,\,\text{such that}\,\, V(|x|) \le C_0 V_p (|x|) \\
    &\text{for any}\,\, x \in B_R^N,\,\,
    \text{where}\,\, V_p (|x|)\,\, \text{is given by Definition \ref{Def V_p}}. \\
    (A3) \quad &\text{There exist} \,\, \la, t_0>0, q>p \,\, \text{such that for any} \,\, |t| \ge t_0,
    F(t) := \int_0^t f(s) \,ds \ge \la |t|^q.\\
    (A4) \quad &\limsup_{t \to 0} \frac{pF(t)}{|t|^p} < \la_V := \inf_{u \in \dot{W}_{0, \rm rad}^{1,p}(B_R^N) \setminus \{0 \} } \frac{\int_{B_R^N} |\nabla u|^p \,dx}{\int_{B_R^N} |u|^p V(|x|) \,dx}.\\
    (A5) \quad &\text{There exist} \,\,\mu >p\,\,\text{and}\,\, t_0 >0 \,\,\text{such that}\,\,\mu F(t) \le f(t)\, t \,\,\text{for any}\,\, |t| \ge t_0.\\
    (A6) \quad &\text{There exist}\,\, t_0>0\,\,\text{and}\,\,M >0\,\,\text{such that} \,\,F(t)  \le M |f(t)| = M |F'(t)|\,\,\text{for any} \,\,|t| \ge t_0.\\
    (A7) \quad &\text{For any}\,\, t \in \re, \,\, p F(t) \le f(t) t.\\
    (A8) \quad &\text{There exists}\,\,C_V >0\,\,\text{such that} \,\, V(|x|) \ge C_V V_p(|x|) \,\, \text{for any}\,\, x \in B_R^N.  \\
    (A9) \quad &\lim_{t \to +\infty} f(t) t e^{-\alpha_0 t^{p'}}  > \frac{p^p}{ \alpha_0^{p-1} C_V L_p}, \,\text{where}\,\,
    L_p := \lim_{n \to \infty} \int_0^1 n e^{n(t^{p'}-t)} \,dt. 
\end{align*}

\begin{remark}\label{Rem A5}
    We can derive (A3) from (A6) and (A1). In fact, by solving the differential inequality for $F$ in (A6), we have $F(t) \ge \la  e^{\frac{|t|}{M}}$ for $|t| \ge t_0$ which implies (A3). 
    Also, we can derive the condition:
    \begin{align*}
        \tilde{(A5)} \quad \text{For any}\,\, \ep >0,\,\text{there exists} \,\, t_{\ep} >0 \,\,\text{such that}\,\, F(t) \le \ep f(t)\, t \,\,\text{for any}\,\, |t| \ge t_{\ep}
    \end{align*}
    from (A6). 
\end{remark}

\begin{remark}\label{Rem la_V}
    The value $\la_V$ in (A4) can be estimated as follows.
    \begin{align*}
        \la_V \ge \begin{cases}
            \( \max_{x \in B_R^N} V(|x|) \)^{-1} \la_p (B_R^N) \quad &{\rm if} \,\, R \in (0, \infty),\\
            \( \max_{x \in B_R^N} |x|^p V(|x|) \)^{-1} \( \frac{N-p}{p} \)^p \quad &{\rm if} \,\, R \in (0, \infty],
        \end{cases}
    \end{align*}
    where 
    \begin{align*}
        \la_p (B_R^N) := \inf_{u \in \dot{W}_{0, \rm rad}^{1,p}(B_R^N) \setminus \{0 \} } \frac{\int_{B_R^N} |\nabla u|^p \,dx}{\int_{B_R^N} |u|^p \,dx},\quad
        \( \frac{N-p}{p} \)^p = \inf_{u \in \dot{W}_{0, \rm rad}^{1,p}(B_R^N) \setminus \{0 \} } \frac{\int_{B_R^N} |\nabla u|^p \,dx}{\int_{B_R^N} \frac{|u|^p}{|x|^p} \,dx}.
    \end{align*}
    If $R=\infty$ and $V(|x|) = V_p (|x|)$, by using Lemma \ref{Lem cpt}, we have
    \begin{align*}
        \int_{\re^N} |u|^p V_p(|x|) \,dx \le  \frac{\Gamma (p)}{p^p} \| \nabla u \|_p^p,
    \end{align*}
    which implies that 
    \begin{align}\label{la_V est}
         \la_{V_p} \ge \frac{p^p}{\Gamma (p)}.
    \end{align}
\end{remark}

\begin{remark}\label{Rem M_p}
    For $L_p$ in (A9), we have the estimate $p \le L_p \le p (p')^{p-1}$ for any $p>1$. In fact, let $t_* = (p')^{-(p-1)}$. Since 
    \begin{align*}
    -t &\le t^{p'} -t \le -\frac{t}{p} \quad \text{for any} \,\, t \in \(0, t_*\),\\
         \frac{t-1}{p-1} &\le t^{p'} -t \le \frac{t-1}{p \( (p')^{p-1} -1 \)} \quad \text{for any} \,\, t \in \(t_* , 1\),
    \end{align*}
    we have
    \begin{align*}
    L_p &\le \lim_{n \to \infty} \int_0^{t_*} n e^{-\frac{nt}{p}} \,dt  + \lim_{n \to \infty} \int_{t_*}^1 n e^{\frac{n(t-1)}{p ( (p')^{p-1} -1 )}} \,dt = p (p')^{p-1},\\
        L_p &\ge \lim_{n \to \infty} \int_0^{t_*} n e^{-nt} \,dt  + \lim_{n \to \infty} \int_{t_*}^1 n e^{\frac{n(t-1)}{p-1}} \,dt = p.
    \end{align*}
\end{remark}


Let $E: \dot{W}_{0, \rm rad}^{1,p} (B_R^N) \to \re$ be the energy functional to (\ref{Elliptic eq}) defined by 
\begin{align}
\label{energy funct}
    E(u) := \frac{1}{p} \int_{B_R^N} |\nabla u |^p \,dx - \int_{B_R^N} F(u) V(|x|) \,dx \quad (\forall u \in \dot{W}_{0, \rm rad}^{1,p} (B_R^N)).
\end{align}
Under (A1) and the assumption on the growth of $f$ (subcritical or critical growth), there exist $\alpha, C_1 >0$ such that 
\begin{align}\label{f upper}
    |f(t)| \le C_1 e^{\alpha |t|^{p'}} \,\,(\forall t \in \re).
\end{align}
Therefore, for any $t \in \re$
\begin{align}\label{F upper}
 0\le F(t) = \int_0^t f(s) \,ds  
    &\le 2 C_1 \int_0^{|t|} e^{\alpha s^{p'}} \,ds \notag \\
    &\le 2 C_1 e^\alpha + 2 C_1 \int_1^{|t|} s^{p'-1} e^{\alpha s^{p'}} \,ds\notag \\
    &= 2C_1 e^\alpha + \frac{2 C_1}{p'\alpha} \( e^{\alpha |t|^{p'}} -e^\alpha  \) 
    \le C_2 e^{\alpha |t|^{p'}}.
\end{align}
By using (\ref{F upper}), (A2) and Theorem \ref{Thm WTM bdd}, we see that 
the functional $E$ is well-defined and is of class $C^1$. 
We show the following results. 

\begin{theorem}\label{Thm subcritical}(Subcritical growth)
    Assume (A1)-(A5) and that $f$ has subcritical growth at $\infty$. Then the equation (\ref{Elliptic eq}) admits a nontrivial radial weak solution.
\end{theorem}

\begin{example}\label{Ex subcritical}(Subcritical growth)
    Let $V(|x|)=V_p(|x|)$. The nonlinearities 
    \begin{align*}
       f_1(t) &=k |t|^{\beta -1} t \,\, (k >0, \beta >p-1)\\
       f_2(t) &= k |t|^{\beta -1} t ( e^{\alpha |t|^{\gamma}} -1) \,\,
        \begin{cases}
        &(i) \,k>0, 0<\gamma <p', \beta >p-1-\gamma, \alpha >0 \,\,\text{or}\\
        &(ii) \,k>0, 0< \gamma<p', \beta = p-1-\gamma, \alpha >0 \,\, \text{with} \,\, k\alpha <  \la_V
        \end{cases} 
    \end{align*}
    satisfy the assumptions of Theorem \ref{Thm subcritical}.
\end{example}

\begin{theorem}\label{Thm critical}(Critical growth)
     Assume (A1), (A2), (A4), (A6)-(A9) and that $f$ has critical growth at $\infty$. Then the equation (\ref{Elliptic eq}) admits a nontrivial radial weak solution.
\end{theorem}

\begin{example}\label{Ex critical}(Critical growth)
    Let $V(|x|)= V_p(|x|)$ and 
\begin{align*}
        f_3(t)&=k |t|^{\beta -1} t (e^{\alpha |t|^{p'}} -1)\quad (k>0, \beta \ge p-1, \alpha >0),\\
        f_4(t) &= k |t|^{\beta -1} t e^{\alpha |t|^{p'}}
         \begin{cases}
        &(i) \,k>0, \beta > -1, \alpha >0 \,\,\text{or}\\
        &(ii) \,k>0, \beta = -1, \alpha >0 \,\, \text{with} \,\, k\alpha^{p-1} > \frac{p^p}{L_p}.
        \end{cases} 
\end{align*}
Then the nonlinearities $f_3(t)$ and 
\begin{align*}
      f_5 (t) =  
      \begin{cases}
        0 \quad &(t \le T),\\
        \frac{f_4 \( (1+\delta)T \)}{\delta T}(t -T) &(T < t < (1+\delta)T),\\
        f_4(t) &(t \ge (1+\delta)T)
        \end{cases} 
\end{align*}
    satisfy the assumptions of Theorem \ref{Thm critical}, where $T>0$ is large enough and $\delta=\min \{ 1, \frac{1}{|p-2|}\}$. 
\end{example}


We use the classical mountain pass theorem by Ambrosetti-Rabinowitz to show Theorem \ref{Thm subcritical} and Theorem \ref{Thm critical}. To use it, we have to check mainly two conditions which are the mountain pass geometry (Lemma \ref{Lem MPG}) and the Palais-Smale condition (Lemma \ref{Lem PS}) for the functional $E$. In the subcritical growth case, we can show that $E$ satisfies the Palais-Smale condition at any level $c \in \re$ in a standard way because of the compactness of the associated embedding. On the other hand, the compactness is lost in the critical growth case. However, in this case, we can show that $E$ satisfies the Palais-Smale condition at  level $c$ which is less than some level $\overline{c}$. This level $\overline{c}$ is so-called (the first) non-compactness level of the functional $E$. First, we will find the non-compactness level $\overline{c}$ of the functional $E$. Next, we will show that the mountain pass level $d$ of $E$:
\begin{align}\label{MP level}
    d&:= \inf_{\gamma \in \Gamma} \max_{t \in [0,1]} E(\gamma (t)),\\
    &\text{where}\,\, \Gamma := \{ \gamma \in C ([0,1]; \dot{W}_{0, \rm rad}^{1,p} (B_R^N)) \,\,; \,\, \gamma (0) =0, \, E(\gamma (1)) < 0 \} \notag 
\end{align}
avoids the non-compactness level $\overline{c}$. Namely, we will show $d<\overline{c}$\, in Lemma \ref{Lem d<c}. We need the assumption (A9) to show Lemma \ref{Lem d<c}. 


\begin{lemma}\label{Lem MPG}
    Assume (A1)-(A4) and that $f$ has subcritical or critical growth at $\infty$. Then the functional $E$ satisfies the following moutain pass geometry.
    \begin{align*}
        &{\rm (i)} \,\,E(0)=0. \\
        &{\rm (ii)} \,\,\text{there exist}\,\, a, \rho >0\,\, \text{such that}\,\, E(u) \ge a \,\, \text{for any}\,\, u \in \dot{W}_{0, \rm rad}^{1,p}(B_R^N) \,\, \text{with}\,\, \| \nabla u\|_p =\rho. \\
        &{\rm (iii)} \,\, \text{there exists}\,\,e_0 \in \dot{W}_{0, \rm rad}^{1,p}(B_R^N)\,\, \text{such that}\,\, E(e_0) <0 \,\, \text{and}\,\,  \| \nabla e_0 \|_p >\rho.
    \end{align*}
\end{lemma}

\begin{proof}{\it ( Lemma \ref{Lem MPG})}

\noindent
(ii) From (A4), there exist $\ep_0 \in (0, 1), t_0 >0$ such that 
\begin{align*}
    F(t) \le \frac{1}{p} \la_V (1- \ep_0)\, |t|^{p} \quad \text{for any}\,\, |t| \le t_0. 
\end{align*}
From (\ref{f upper}), for $r \in (p, p^*)$ there exists $C_1 >0$ such that 
\begin{align*}
    F(t) \le C_1 |t|^r e^{\alpha |t|^{p'}} \quad \text{for any}\,\, |t| \ge t_0. 
\end{align*}
Therefore, we have
\begin{align*}
    E(u) &\ge \frac{1}{p} \int_{B_R^N} |\nabla u|^p \,dx - \frac{\la_V (1-\ep_0)}{p} \int_{B_R^N} |u|^p V(|x|) \,dx - C_1 \int_{B_R^N} |u|^r e^{\alpha |u|^{p'}} \, V(|x|) \,dx\\
    &\ge \frac{\ep_0}{p} \int_{B_R^N} |\nabla u|^p \,dx - C_1 \( \int_{B_R^N} |u|^{p^*} V(|x|) \,dx \)^{\frac{r}{p^*}} \( \int_{B_R^N} e^{\frac{p^* \alpha}{p^* -r} |u|^{p'}} V(|x|) \,dx \)^{1-\frac{r}{p^*}}.
\end{align*}
From Theorem \ref{Thm WTM bdd} and (A2), for any $u \in \dot{W}_{0, \rm rad}^{1,p}(B_R^N)$ with $\| \nabla u \|_p =\sigma$ we have
\begin{align*}
    \int_{B_R^N} e^{\frac{p^* \alpha}{p^* -r} |u|^{p'}} V(|x|) \,dx 
    \le C \int_{B_R^N} e^{\frac{p^* \alpha}{p^* -r} \sigma^{p'} \( \frac{|u|}{\sigma} \)^{p'}} V_p(|x|) \,dx \le C_3,
\end{align*}
where $\sigma >0$ satisfies $\frac{p^* \alpha}{p^* -r} \sigma^{p'} \le \alpha_p$. From the Sobolev inequality, we have
\begin{align*}
    E(u) \ge \frac{\ep_0}{p} \| \nabla u\|_p^p -C_4 \| \nabla u \|_p^r = \frac{\ep_0}{p} \sigma^p -C_4 \sigma^r
\end{align*}
for $\sigma = \| \nabla u \|_p \le \( \alpha_p \frac{p^* -r}{p^* \alpha} \)^{\frac{1}{p'}}$. 
Note that the function $g(\sigma):= \frac{\ep_0}{p} \sigma^p -C_4 \sigma^r$ is increasing on the interval $\(0,  \( \frac{\ep_0}{C_4 r} \)^{\frac{1}{r-p}} \)$. Set $\rho := \min \left\{ \( \alpha_p \frac{p^* -r}{p^* \alpha} \)^{\frac{1}{p'}},  \( \frac{\ep_0}{C_4 r} \)^{\frac{1}{r-p}} \right\}$ and $a:= g(\rho)$. Then we get (ii).

\noindent
(iii) From (A3), we have
$F(t) \ge C_5 |t|^q -C_5$ for any $t \in \re$ and for some $q >p$. From (A2), there exists an open set $U \subset B_R^N$ such that $V(|x|) \ge 0$ for any $x \in U$. 
Let $u_0 \in C_{c, \rm rad}^\infty (U) \setminus \{ 0\}$. Then for $t >0$ we have
\begin{align*}
    E(t u_0) &\le \frac{\| \nabla u_0 \|_p^p}{p} t^p - C_5 \(  \int_{B_R^N} |u_0 |^q V(|x|) \,dx \) t^q  + C_5 \( \int_{B_R^N} V(|x|) \,dx\)\\
    &\to - \infty \quad (t \to + \infty).
\end{align*}
For large $t_0 >0$, set $e_0 = t_0 u_0$. Then we get (iii). 
\end{proof}

\begin{remark}
Lemma \ref{Lem MPG} implies that $d>0$ under the assumptions (A1)-(A4).
\end{remark}


We say that $E$ satisfies the Palais-Smale condition at level $c$, for short, the $(PS)_c$-condition, if for $(PS)_c$-sequence $\{ u_m \}_{m \in \N} \subset \dot{W}_{0, \rm rad}^{1,p}(B_R^N)$, i.e. $E(u_m) \to c$ and $E'(u_m) \to 0$, there exists a strongly convergent subsequence in $\dot{W}_{0, \rm rad}^{1,p}(B_R^N)$. 

\begin{lemma}\label{Lem PS}

\noindent
    {\rm (I)} Assume (A1), (A2), (A5), and that $f$ satisfy subcritical growth at $\infty$. Then $E$ satisfies the $(PS)_c$-condition for any $c \in \re$.
    
\noindent
    {\rm (II)} Assume (A1), (A2), (A6), (A7), and that $f$ satisfy critical growth at $\infty$. Then $E$ satisfies the $(PS)_c$-condition for any $c < \overline{c} := \frac{1}{p} \( \frac{\alpha_p}{\alpha_0} \)^{p-1}$, where $\alpha_0$ and $\alpha_p$ are given by (\ref{critical cond}) and Theorem \ref{Thm WTM bdd} (II), respectively.
\end{lemma}

\begin{proof}{\it ( Lemma \ref{Lem PS})} 
Let $\{ u_m \}_{m \in \N} \subset \dot{W}_{0, \rm rad}^{1,p}(B_R^N)$ be a $(PS)_c$-sequence, namely
\begin{align}\label{E to c}
    E(u_m) &= \frac{1}{p} \| \nabla u_m \|^p_p - \int_{B_R^N} F(u_m) V(|x|)\,dx \to c\,\,(m \to \infty),\\
    \label{E' to 0}
    |\,E'(u_m) [\varphi]\,| &= \left| \,\int_{B_R^N} |\nabla u_m|^{p-2} \nabla u_m \cdot \nabla \varphi - f(u_m) \varphi V(|x|) \,dx \,\right| \le \ep_m \| \nabla \varphi \|_p, 
\end{align}
where $\ep_m \to 0$ as $m \to \infty$. Here $E'(u)$ is the Fr\'echet derivative of $E$ at $u$. From (A5), (\ref{E to c}) and (\ref{E' to 0}), we have
\begin{align*}
    \| \nabla u_m \|_p^p &= p E(u_m) + p \int_{B_R^N} F(u_m) V(|x|) \,dx\\
    &\le pc + p \int_{\{ |u_m|\le t_0 \}} F(u_m) V(|x|) \,dx + \frac{p}{\mu} \int_{B_R^N} f(u_m) u_m V(|x|) \,dx+ o(1)\\
    &\le C + \frac{p}{\mu} \,\| \nabla u_m \|_p^p + \frac{p}{\mu} \, \ep_m \| \nabla u_m \|_p + o(1),\ \ \ \text{as}
\ m\rightarrow\infty.
\end{align*}
If $\| \nabla u_m \|_p$ is not bounded, it contradicts the above inequality because $\mu >p$. Therefore, $\| \nabla u_m \|_p \le K$ for some $K>0$ and any $m \in \N$. Then there exists $u_* \in \dot{W}_{0, {\rm rad}}^{1,p}(B_R^N)$ such that
\begin{align*}
    &u_m \rightharpoonup u_* \,\,{\rm in} \,\, \dot{W}_{0, {\rm rad}}^{1,p}(B_R^N), \,\,
    \|\nabla u_* \|_p \le \liminf_{m \to \infty} \|\nabla u_m \|_p \le K,\\
    &u_m \to u_* \,\,\text{a.e. in}\,\,B_R^N, \,\, u_m \to u_* \,\,{\rm in}\,\, L^q(B_R^N) \,\,\text{for any}\,\, q \in (p, p^*).
\end{align*}
Here, we use the compactness of the embedding $\dot{W}_{0, {\rm rad}}^{1,p}(B_R^N) \hookrightarrow L^q(B_R^N)$ by Strauss. 
Again, from (\ref{E' to 0}), we have
\begin{align}\label{f(u_m)u_m}
   \| \nabla u_m \|_p^p &= \int_{B_R^N} f(u_m) u_m V(|x|) \,dx + o(1),\\
    \label{f(u_m)u}
     \int_{B_R^N} |\nabla u_m|^{p-2} \nabla u_m \cdot \nabla u_* \,dx &= \int_{B_R^N} f(u_m) u_* V(|x|) \,dx + o(1).
\end{align}
It is enough to show that there exist $C>0$ and $a >1$ such that for any $m \in \N$  
\begin{align}\label{L^a bdd}
    \int_{B_R^N} |f(u_m)|^a V(|x|) \,dx < C. 
\end{align}
In fact, if (\ref{L^a bdd}) holds, by using (A2) and Lemma \ref{Lem cpt}, we have
\begin{align*}
   &\left| \int_{B_R^N} f(u_m) u_m V(|x|) \,dx - \int_{B_R^N} f(u_m) u_* V(|x|) \,dx \,\right| \\
   &\le C \( \int_{B_R^N} |f(u_m)|^a V(|x|) \,dx \)^{\frac{1}{a}} \( \int_{B_R^N} |u_m - u_*|^{\frac{a}{a-1}} V_p(|x|) \,dx \)^{1-\frac{1}{a}}
   = o(1)
\end{align*}
which implies that $\int_{B_R^N} |\nabla u_m|^{p-2} \nabla u_m \cdot \nabla u_* \,dx = \int_{B_R^N} |\nabla u_m|^{p} \,dx + o(1)$. Therefore,
\begin{align*}
    \lim_{m \to \infty} \int_{B_R^N} \( |\nabla u_m|^{p-2} \nabla u_m - |\nabla u_*|^{p-2} \nabla u_* \) \cdot (\nabla u_m - \nabla u_*) \,dx =0.
\end{align*}
By using the following inequality for $a, b \in \re^N$
\begin{align*}
    \( |b|^{p-2} b - |a|^{p-2} a \) \cdot (b-a) \ge \begin{cases}
        2^{2-p} |b-a|^p \quad &{\rm if}\,\,p \ge 2,\\
        (p-1) |b-a|^2 \( |a|^2 + |b|^2 \)^{\frac{p-2}{2}} &{\rm if}\,\, 1 < p \le 2,
    \end{cases} 
\end{align*}
we have $u_m \to u_*$ in $\dot{W}_{0, {\rm rad}}^{1,p}(B_R^N)$, which means that $E$ satisfies the $(PS)_c$-condition. From now on, we shall show (\ref{L^a bdd}). 

\noindent
(I) Since $f$ has subcritical growth at $\infty$, there exists $t_0 >0$ such that for any $|t| \ge t_0$ 
\begin{align*}
    |f(t)| \le  \exp \( \frac{ \alpha_p}{p' K^{p'}} |t|^{p'} \).
\end{align*}
From Theorem \ref{Thm WTM bdd}, we have
\begin{align*}
     \int_{B_R^N} |f(u_m)|^{p'} V(|x|) \,dx 
     &\le C +   \int_{B_R^N}  \exp \( \frac{\alpha_p}{K^{p'}} |u_m|^{p'} \) V(|x|)\,dx  \\
     &\le C + C \int_{B_R^N} \exp \( \alpha_p \( \frac{|u_m|}{\| \nabla u_m \|_p } \)^{p'} \) V_p (|x|) \,dx < C,
\end{align*}
which implies (\ref{L^a bdd}). Therefore, $E$ satisfies the $(PS)_c$-condition for any $c \in \re$.


\noindent
(II) 
By using $u_m \to u_*$ in $L^q (B_R^N)$ for any $q \in (p, p^*)$ and (\ref{f(u_m)u_m}), we have
\begin{align*}
    \int_{B_R^N} f(u_m) u_m V(|x|) \,dx \le K^{p}+1 \,\, \text{and} \,\, f(u_m), f(u_*) \in L^1 (B_R^N; V(|x|) \,dx).
\end{align*}
Then we get the following. We will show it later.
\begin{lemma}\label{Lem f(u_m)}
$f(u_m) \to f(u_*)$ in $L^1 (B_R^N; V(|x|) \,dx )$. 
\end{lemma}
 
\noindent
Therefore, for any $\varphi \in \dot{W}_{0, {\rm rad}}^{1,p}(B_R^N)$
\begin{align}\label{weak form for smooth funct}
    \lim_{m \to \infty} \int_{B_R^N} |\nabla u_m|^{p-2} \nabla u_m \cdot \nabla \varphi \,dx = \int_{B_R^N} f(u_*) \varphi V(|x|) \,dx.
\end{align}
In fact, from (\ref{E' to 0}), we see that the equality (\ref{weak form for smooth funct}) holds for any $\varphi \in C_{c, \rm rad}^\infty (B_R^N)$. By the density argument, the equality (\ref{weak form for smooth funct}) also holds for any $\varphi \in \dot{W}_{0, {\rm rad}}^{1,p}(B_R^N)$, because, if $\varphi_m \to \varphi$ in $\dot{W}_{0, {\rm rad}}^{1,p}(B_R^N)$, 
\begin{align*}
    \int_{B_R^N} f(u_*) |\varphi_m - \varphi | V(|x|) \,dx 
    &\le \( \int_{B_R^N} |\varphi_m - \varphi |^{p^*} V(|x|) \,dx \)^{\frac{1}{p^*}} \( \int_{B_R^N} f(u_*)^{\frac{p^*}{p^* -1}} V(|x|) \,dx \)^{1-\frac{1}{p^*}} \\
    &\le C \,\| \nabla (\varphi_m - \varphi) \|_p \( C + \int_{B_R^N} e^{ (\alpha_0 +1) \frac{p^* }{p^* -1} |u_*|^p} V_p(|x|) \,dx \)^{1-\frac{1}{p^*}}\\
    &\to 0 \,\, (m \to \infty).
\end{align*}

\noindent
Since 
\begin{align*}
    0 \le  F(u_m (x)) \le M |f(u_m (x))|\,\, \text{for any}\,\, x \in \{ x \in B_R^N \,\,|\,\, |u_m (x)| \ge t_0 \}
\end{align*}
from (A6), 
the generalized Lebesgue dominated convergence theorem (see e.g. Remark in p.20 in \cite{LL}) implies that 
$F(u_m) \to F(u_*)$ in $L^1 (B_R^N; V(|x|) \,dx )$. 
Therefore, from (\ref{E to c}), we see that
\begin{align}\label{u_m lim}
    \lim_{m \to \infty} \| \nabla u_m \|^p_p = p \( c + \int_{B_R^N} F(u_*) V(|x|) \,dx\).
\end{align}

\noindent
We divide three cases.

\noindent
(II)-(i) {\bf The case where $0= \| \nabla u_* \|_p < \lim_{m \to \infty} \| \nabla u_m \|_p$.} We will show (\ref{L^a bdd}). From (\ref{u_m lim}) and (\ref{f(u_m)u_m}), we have
\begin{align*}
   \( 0, \, \( \frac{\alpha_p}{\alpha_0} \)^{p-1} \) \ni pc = \lim_{m \to \infty} \| \nabla u_m \|_p^p = \lim_{m \to \infty} \int_{B_R^N} f(u_m) u_m V(|x|) \,dx.
\end{align*}
Then there exist $a=a(c) >1$ and $\delta_a >0$ such that $a (\alpha_0 + \delta_a) (pc)^{\frac{1}{p-1}} < \alpha_p$. Also, there exists $m_a \in \N$ such that $a (\alpha_0 + \delta_a) \| \nabla u_m \|_p^{p'} \le \alpha_p$ for any $m \ge m_a$. Since $f$ has critical growth at $\infty$, there exists $t_a >0$ such that $|f(t)| \le e^{(\alpha_0 + \delta_a)|t|^{p'}}$ for any $|t| \ge t_a$. Therefore, we have
\begin{align*}
    \int_{B_R^N} |f(u_m)|^a V(|x|) \,dx 
    &\le C + \int_{B_R^N} e^{a(\alpha_0 + \delta_a) |u_m|^{p'}} V(|x|) \,dx\\
    &\le C + C \int_{B_R^N} \exp \( \alpha_p \frac{|u_m|^{p'}}{\| \nabla u_m \|_p^p} \) V_p(|x|) \,dx \le C,
\end{align*}
which implies (\ref{L^a bdd}). Therefore, $E$ satisfies the $(PS)_c$-condition for any $c < \frac{1}{p} \( \frac{\alpha_p}{\alpha_0} \)^{p-1}$.

\noindent
(II)-(ii) {\bf The case where $\| \nabla u_* \|_p = \lim_{m \to \infty} \| \nabla u_m \|_p$.} In this case, we get $u_m \to u_*$ in $\dot{W}_{0, {\rm rad}}^{1,p}(B_R^N)$, which means that $E$ satisfies the $(PS)_c$-condition for any $c < \frac{1}{p} \( \frac{\alpha_p}{\alpha_0} \)^{p-1}$. 

\noindent
(II)-(iii) {\bf The case where $0< \| \nabla u_* \|_p < \lim_{m \to \infty} \| \nabla u_m \|_p$.} 

We will show 
(\ref{L^a bdd}) following the argument in \cite{Panda, SS}. Let $\mathcal{M} (\overline{B_R^N})$ be the space of all Borel regular measures on $\overline{B_R^N}$. By the weak compactness of measures (see e.g. p.55 in \cite{EG}), there exist $\mu_1, \mu_2 \in \mathcal{M}(\overline{B_R^N})$ such that 
\begin{align*}
 |\nabla u_m |^p \,dx \overset{*}{\rightharpoonup}  d\mu_1,\, \,\,f(u_m) u_m V \,dx \overset{*}{\rightharpoonup} d\mu_2 \quad \text{in}\,\, \mathcal{M} (\overline{B_R^N}),\ \ \ \text{as}\ m\rightarrow\infty.
\end{align*}
Here, $\mu_m \overset{*}{\rightharpoonup} \mu$ in $\mathcal{M} (\overline{B_R^N})$ means that for any $f \in C_c(\overline{B_R^N})$, $\int_{B_R^N} f d\mu_m \to \int_{B_R^N} f d\mu$. 
Let $T$ satisfy $|\nabla u_m |^{p-2} \nabla u_m \rightharpoonup T \,\,{\rm in}\,\, \( L^{p'} (\re^N) \)^N$. 
First, we claim that there exists $\delta_0 >0$ such that for any $\delta \in (0, \delta_0 ]$
\begin{align}\label{mu_1 < level}
    \mu_1 \( B_{\delta}^N \) < \( \frac{\alpha_p}{\alpha_0} \)^{p-1}.
\end{align}
We will show it later. 
For $\delta >0$, let $\psi_\delta \in C_{c, \rm rad}^1 (B_{\delta}^N)$ satisfy $0\le \psi_\delta \le 1, \psi_\delta (x) =1$ for $x \in B^N_{\delta /2}$ and $|\nabla \psi_\delta | \le \overline{C} \delta^{-1}$ for a constant $\overline{C}>0$.
Next, we claim that for any $\ep \in \left(0, \frac{1}{3}- \frac{1}{3} \(\frac{\alpha_0}{\alpha_p}\)^{p-1} \mu_1 (B^N_{\delta_0}) \right]$, there exists $\delta_1 \in (0, \delta_0 ]$ such that 
\begin{align}\label{u_m psi}
    \int_{B_{\delta_1}^N} |\nabla (u_m \psi_{\delta_1} ) |^p\,dx \le (1-\ep) \(\frac{\alpha_p}{\alpha_0} \)^{p-1}.
\end{align}
We will show it later. If we set 
\begin{align*}
    v_m := (1-\ep)^{-\frac{1}{p}} \( \frac{\alpha_0}{\alpha_p}\)^{\frac{p-1}{p}} u_m \psi_{\delta_1},
\end{align*}
then we have $\| \nabla v_m \|_{L^p (B_{\delta_1}^N)} \le 1$. By using the weighted TM inequality in Theorem \ref{Thm WTM bdd} for $\{ v_m \}_{m \in \N} \subset \dot{W}_{0, \rm rad}^{1,p} (B^N_{\delta_1})$, we have
\begin{align*}
    \sup_{m \in \N} \int_{B_{\delta_1/2}^N} e^{\alpha_0 (1-\ep)^{-\frac{1}{p-1}} |u_m|^{p'} } V_p (|x|) \,dx
    &\le \sup_{m \in \N} \int_{B_{\delta_1}^N} e^{\alpha_0 (1-\ep)^{-\frac{1}{p-1}} |u_m \psi_{\delta_1}|^{p'}} V_p (|x|) \,dx\\
    &= \sup_{m \in \N} \int_{B_{\delta_1}^N} e^{\alpha_p |v_m |^{p'}} V_p (|x|) \,dx< \infty.
\end{align*}
In the same way as the case (II)-(i), there exist $a = a(\ep)>1$ and $\delta_a, t_a >0$ such that $a (\alpha_0 + \delta_a) \le \alpha_0 (1- \ep)^{-\frac{1}{p-1}}$ and $|f(t)| \le e^{(\alpha_0 + \delta_a ) |t|^{p'}}$ for any $|t| \ge t_a$. 
Finally, we have
\begin{align*}
     \int_{B_R^N} |f(u_m)|^a V (|x|)\,dx 
   &\le C+ C \int_{B^N_{\delta_1/2}} e^{a(\alpha_0 + \delta_a ) |u_m|^{p'}} V_p (|x|) \,dx\\
     &\le C + C \int_{B^N_{\delta_1 /2 }} e^{\alpha_0 (1-\ep)^{-\frac{1}{p-1}} |u_m|^{p'}} V_p (|x|) \,dx \le C.
\end{align*}
Therefore, we get (\ref{L^a bdd}). Therefore, $E$ satisfies the $(PS)_c$-condition for any $c < \frac{1}{p} \( \frac{\alpha_p}{\alpha_0} \)^{p-1}$.
which implies $u_m \to u_*$ in $\dot{W}_{0, {\rm rad}}^{1,p}(B_R^N)$. Therefore, $E$ satisfies the $(PS)_c$-condition for any $c < \frac{1}{p} \( \frac{\alpha_p}{\alpha_0} \)^{p-1}$. 
\end{proof}




\begin{proof}{\it ( Proof of (\ref{mu_1 < level}))} 
For $\delta \in (0,\frac{R}{2})$, let $\varphi_\delta \in C_{\rm rad}^1 (B_R^N)$ satisfy $0\le \varphi_\delta \le 1, \varphi_\delta (x) =1$ for $x \in B_R^N \setminus B^N_{2\delta}$ and $\varphi_\delta (x) =0$ for $x \in B^N_{\delta}$. 
Since $\lim_{m \to \infty} E'(u_m) [u_m \varphi] = \lim_{m \to \infty} E'(u_m) [u_* \varphi] =0$ for any $\varphi \in C_{\rm rad}^1 (B_R^N)$, we have
\begin{align*}
    0 &
    = \lim_{m \to \infty} \left[ \int_{B_R^N} |\nabla u_m |^p \varphi_\delta + |\nabla u_m|^{p-2} (\nabla u_m \cdot \nabla \varphi_\delta ) u_m - f(u_m) u_m \varphi_\delta V(|x|) \,dx \right]\\
    &= \int_{B^N_R} \varphi_\delta\, d\mu_1 + \int_{B_R^N} u_* (T \cdot \nabla \varphi_\delta ) \,dx - \int_{B_R^N} \varphi_\delta \,d\mu_2, \\
    0 &=  \lim_{m \to \infty} \left[ \int_{B_R^N}  |\nabla u_m|^{p-2} \nabla u_m \cdot \nabla (u_* \varphi_\delta )  - f(u_m) u_* \varphi_\delta V(|x|) \,dx \right]\\
    &=  \int_{B^N_R} \varphi_\delta \, (T \cdot \nabla u_*) \,dx + \int_{B_R^N} u_* (T \cdot \nabla \varphi_\delta ) \,dx - \int_{B_R^N} f(u_*) u_* \varphi_\delta V(|x|) \,dx
\end{align*}
which imply that
\begin{align}\label{mu_1 to mu_2}
    \int_{B^N_R} \varphi_\delta \,d\mu_1 - \int_{B_R^N} \varphi_\delta \,d\mu_2 
    = \int_{B_R^N} \varphi_\delta (T \cdot \nabla u_* ) \,dx - \int_{B_R^N} f(u_*) u_* \varphi_\delta V(|x|) \,dx.
\end{align}
Also, by using (\ref{weak form for smooth funct}) with $\varphi=u_*\in \dot{W}^{1,p}_{0,\text{rad}}(B^N_R)$, we have
\begin{align}\label{T to f}
    \int_{B_R^N}  (T \cdot \nabla u_* ) \,dx = \int_{B_R^N} f(u_*) u_* V(|x|) \,dx.
\end{align}
From (\ref{u_m lim}), (\ref{mu_1 to mu_2}), (\ref{T to f}) and (A7), we have
\begin{align*}
    \( \frac{\alpha_p}{\alpha_0} \)^{p-1} &> pc 
    = \mu_1 (B_R^N) - p \int_{B_R^N} F(u_*) V(|x|) \,dx\\
    &\ge \mu_1 (B_{\delta}^N) + \mu_1 (B_R^N \setminus B_\delta^N) - \int_{B_R^N} f(u_*) u_* V(|x|) \,dx\\
    &\ge \mu_1 (B_{\delta}^N) + \int_{B_R^N} \varphi_\delta \,d\mu_1 - \int_{B_R^N} f(u_*) u_* \varphi_\delta V(|x|) \,dx 
    - \int_{B_{2\delta}^N} f(u_*) u_* V(|x|) \,dx\\
    &= \mu_1 (B_{\delta}^N) + R_1(\delta) + R_2 (\delta) + R_3 (\delta), \\
    R_1(\delta)&= \int_{B_R^N} \varphi_\delta \,d\mu_2  - \int_{B_R^N} f(u_*) u_* \varphi_\delta V(|x|) \,dx\\
    &= \lim_{m \to \infty} \int_{B_R^N} f(u_m) u_m \varphi_\delta V(|x|) \,dx - \int_{B_R^N} f(u_*) u_* \varphi_\delta V(|x|) \,dx =0,\\
    R_2 (\delta) &= \int_{B_R^N} \varphi_\delta (T \cdot \nabla u_* ) \,dx  - \int_{B_R^N} f(u_*) u_* \varphi_\delta V(|x|) \,dx \\
    &\to \int_{B_R^N}  (T \cdot \nabla u_* ) \,dx  - \int_{B_R^N} f(u_*) u_*  V(|x|) \,dx= 0\quad (\delta \to 0),\\
    R_3 (\delta) &= - \int_{B_{2\delta}^N} f(u_*) u_* V(|x|) \,dx \to 0 \quad (\delta \to 0),
\end{align*}
where $R_1 (\delta) =0$ comes from the radial lemma: $|u_m (x)|\le C \delta^{-\frac{N-p}{p}}$ a.e. $x \in B_R^N \setminus B_\delta^N$ and the Lebesgue dominated convergence theorem. Therefore, we get (\ref{mu_1 < level}). 
\end{proof}


\begin{proof}{\it ( Proof of (\ref{u_m psi}))} 
Since for any $\delta \in (0, \delta_0]$
\begin{align*}
    \lim_{m\to \infty} \int_{B^N_{\delta}} |\nabla u_m |^p \psi_\delta^p \,dx 
    = \int_{B^N_{\delta}} \psi_\delta^p \,d\mu_1 
    \le \mu_1 (B^N_{\delta}) \le \mu_1 (B^N_{\delta_0}) 
    \le (1-3\ep) \( \frac{\alpha_p}{\alpha_0} \)^{p-1}
\end{align*}
and $|a+b|^p \le |a|^p + p |a+b|^{p-1} |b|$ for any $a, b \in \re^N$ and $p > 1$, we have
\begin{align}\label{u_m psi_delta}
    &\lim_{m \to \infty} \int_{B_\delta^N} |\nabla (u_m \psi_\delta) |^p \,dx \notag  \\
    &\le \lim_{m \to \infty} \left[ \int_{B_\delta^N} |\nabla u_m |^p \psi_\delta^p  \,dx + p \int_{B_\delta^N \setminus B_{\delta /2}^N} |\psi_\delta \nabla u_m + u_m \nabla \psi_\delta |^{p-1} |\nabla \psi_\delta |\, |u_m| \,dx \right] \notag \\
    &\le (1-3\ep) \( \frac{\alpha_p}{\alpha_0} \)^{p-1} + p \max \{ 1, 2^{p-2} \} \lim_{m \to \infty} (A(m, \delta) + B(m , \delta)),
\end{align}
where
\begin{align}\label{A(m, delta)}
    A(m, \delta) &:= \int_{B_\delta^N \setminus B_{\delta /2}^N} \psi_\delta^{p-1} |\nabla u_m|^{p-1} |\nabla \psi_\delta |\, |u_m| \,dx
   \le K^{p-1} B(m, \delta)^{\frac{1}{p}},\\
    \label{B(m, delta)}
    B(m, \delta) &:= \int_{B_\delta^N \setminus B_{\delta /2}^N}  |u_m|^{p} |\nabla \psi_\delta |^p \,dx \le \overline{C}^p |B_1^N |^{1-\frac{p}{p^*}} \( \int_{B_\delta^N \setminus B_{\delta /2}^N}  |u_m|^{p^*} \,dx \)^{\frac{p}{p^*}}.
\end{align}
Note that $p^* := \frac{Np}{N-p}$ is the Sobolev critical exponent and $u_* \in \dot{W}^{1,p}_{0, \rm rad}(B_R^N) \subset L^{p^*}(B_R^N)$ by the Sobolev inequality. Since $|u_m (x)| \le C \delta^{-\frac{N-p}{p}}$ for a.e. $x \in B_{\delta}^N \setminus B^N_{\delta /2}$ and $u_m \to u_*$ a.e. in $B_R^N$, the Lebesgue dominated convergence theorem implies that for fixed $\delta >0$, 
\begin{align}\label{u_m^p^*}
    \lim_{m \to \infty} \int_{B_\delta^N \setminus B_{\delta /2}^N}  |u_m|^{p^*} \,dx = \int_{B_\delta^N \setminus B_{\delta /2}^N}  |u_*|^{p^*} \,dx.
\end{align}
Now we choose $\delta_1 \in (0, \delta_0]$ which satisfies 
\begin{align}\label{A est}
p \max \{ 1, 2^{p-2} \} \, \overline{C}^p |B_1^N|^{1-\frac{p}{p^*}} \(\int_{B_{\delta_1}^N \setminus B_{\delta_1 /2}^N}  |u_*|^{p^*} \,dx \)^{\frac{p}{p^*}} 
&\le \ep \( \frac{\alpha_p}{\alpha_0} \)^{p-1},\\
\label{B est}
p \max \{ 1, 2^{p-2} \} \,\overline{C} \,|B_1^N|^{\frac{1}{p}-\frac{1}{p^*}} \(\int_{B_{\delta_1}^N \setminus B_{\delta_1 /2}^N}  |u_*|^{p^*} \,dx \)^{\frac{1}{p^*}} &\le \ep \( \frac{\alpha_p}{\alpha_0} \)^{p-1}.
\end{align}
By using (\ref{u_m psi_delta}), (\ref{A(m, delta)}), (\ref{B(m, delta)}), (\ref{u_m^p^*}), (\ref{A est}) and (\ref{B est}), we get (\ref{u_m psi}).
\end{proof}


\begin{proof}{\it ( Lemma \ref{Lem f(u_m)})} 
We follow the argument in the proof of \cite[Lemma 2.1]{FMR}. First, let $R< \infty$. Since $L^q (B_R^N) \subset L^1 (B_R^N)$ for $q >1$, we see that $u_m \to u_*$ in $L^1 (B_R^N)$ and $u_m \to u_*$ a.e. in $B_R^N$. For any $\ep >0$, there exist $M >0$ and $m_\ep \in \N$ such that for any $m \ge m_\ep$
\begin{align*}
    I_1 (m, M) &:= \int_{\{ |u_m| > M \}} |f (u_m) - f(u_*) | V(|x|) \,dx < \frac{\ep}{2}, \\
    I_2 (m, M) &:= \int_{\{ |u_m| \le  M \} }  |f (u_m) - f(u_*) | V(|x|) \,dx < \frac{\ep}{2}.
\end{align*}
In fact, from the radial lemma (\ref{Radial Lemma}), we have
\begin{align*}
    |u_m(x) | \le \( \frac{p-1}{N-p} \)^{\frac{p-1}{p}} \w_{N}^{-\frac{1}{p}} K|x|^{-\frac{N-p}{p}}
\end{align*}
which implies that 
\begin{align*}
    \{ |u_m| > M \} \subset B^N_{K(M)}, \,\,\text{where}\,\, K(M):= \left[ \frac{K}{M \w_N^{\frac{1}{p}}} \( \frac{p-1}{N-p}\)^{\frac{p-1}{p}} \right]^{\frac{p}{N-p}} \to 0 \,\, (M \to \infty).
\end{align*}
Therefore, 
\begin{align*}
    I_1 (m, M) 
    &\le \int_{\{ |u_m| > M \}} |f (u_m)| V (|x|) \,dx  + \int_{\{ |u_m| > M \}} |f(u_*)| V(|x|) \,dx\\
    &\le \frac{1}{M} \int_{\{ |u_m| > M \}} f (u_m) u_m V (|x|) \,dx  + \int_{\{ |u_m| > M \}} |f(u_*)| V(|x|) \,dx\\
    &\le \frac{K^p +1}{M} + \int_{B_{K(M)}^N} |f(u_*)| V(|x|) \,dx \to 0\,\,(M \to \infty).
\end{align*}
From the Lebesgue dominated convergence theorem, we have
\begin{align*}
    I_2 (m, M) \to 0 \,\, (m \to \infty) \quad \text{for fixed}\,\,M>0. 
\end{align*}
Therefore, $\lim_{m \to \infty} \int_{B_R^N} |f (u_m) - f(u_*) | V(|x|) \,dx =0$. Next, let $R=\infty$, i.e. $B^N_{\infty} = \re^N$. In the same way as above, we get $\lim_{m \to \infty} \int_{B_1^N} |f (u_m) - f(u_*) | V(|x|) \,dx =0$. For any $x \in \re^N \setminus B_1^N$, we have
\begin{align*}
    |u_m(x)|  \le  \( \frac{p-1}{N-p} \)^{\frac{p-1}{p}} \w_{N}^{-\frac{1}{p}} K =: \overline{K}.
\end{align*}
Therefore, we have
\begin{align*}
   \lim_{m \to \infty} \int_{\re^N} |f(u_m) - f(u_*) | V(|x|) \,dx 
   = \lim_{m \to \infty} \int_{\{ |u_m| \le \overline{K} \} } |f(u_m) - f(u_*) | V(|x|) \,dx =0,
\end{align*}
by using the Lebesgue dominated convergence theorem. 
\end{proof}


\begin{lemma}\label{Lem d<c}
    Let $d$ be given by (\ref{MP level}). Assume (A1), (A2), (A8), (A9), and that $f$ satisfy critical growth at $\infty$. Then $d< \overline{c} = \frac{1}{p} \( \frac{\alpha_p}{\alpha_0} \)^{p-1}$. 
\end{lemma}

\begin{proof}{\it ( Lemma \ref{Lem d<c})}
Let $u_k \in \dot{W}_{0, \rm rad}^{1,p} (B_R^N)$ be given by the proof of Theorem \ref{Thm WTM bdd} (II). 
Since 
\begin{align*}
    d \le \max \left\{ E(tu_k)\,\,|\,\, t \ge 0 \right\}  \quad ({\forall}k \in \N),
\end{align*}
it is enough to show that there exists $k \in \N$ such that
\begin{align}\label{Claim d<c}
    \max \left\{ E(t u_k)\,\,|\,\, t \ge 0 \right\} = E(t_k u_k) < \frac{1}{p} \( \frac{\alpha_p}{\alpha_0} \)^{p-1}.
\end{align}
We show (\ref{Claim d<c}) by deriving a contradiction. Suppose that for any $k \in \N$
\begin{align*}
    \frac{t_k^p}{p} \| \nabla u_k \|_p^p - \int_{B_R^N} F(t_k u_k) V(|x|) \,dx \ge\frac{1}{p} \( \frac{\alpha_p}{\alpha_0} \)^{p-1}. 
\end{align*}
Since $\| \nabla u_k \|_p =1$ and $F(u), V(|x|) \ge 0$, we see that
\begin{align}\label{t_k lower}
    t_k^{p'} \ge \frac{\alpha_p}{\alpha_0}   \quad ({\forall}k \in \N).
\end{align}
By using (A9), for any $\ep >0$ there exists $T_{\ep} >0$ such that 
\begin{align*}
    f(t) t \ge (\beta -\ep ) e^{\alpha_0 t^{p'}} \quad ({\forall}t \ge T_\ep), 
\end{align*}
where $\beta = \lim_{t \to + \infty} f(t) t e^{-\alpha_0 t^{p'}}$. Since 
\begin{align*}
    0 = \left. \frac{d}{dt} \right|_{t=t_k} E(tu_k) 
    = t_k^{p-1} - \int_{B_R^N} f(t_k u_k) u_k V(|x|) \,dx,
\end{align*}
for large $k \in \N$ we have
\begin{align*}
    t_k^p &\ge C_V \int_{B_{r_k}^N} f(t_k u_k) t_k u_k V_p(|x|) \,dx\\
    &=(\beta -\ep) C_V \exp \( \alpha_0 t_k^{p'} k \w_p^{-\frac{1}{p-1}} \) \( \int_{B_{r_k}^N} V_p (|x|) \,dx\)\\
    &= \frac{(\beta -\ep) C_V \w_p }{p} \exp \left[ \alpha_0 \w_p^{-\frac{1}{p-1}} k  \(t_k^{p'} - \frac{\alpha_p}{\alpha_0}\) \right]
\end{align*}
which implies that $t_k$ is bounded and $\lim_{k \to \infty} t_k^{p'} = \frac{\alpha_p}{\alpha_0}$ by (\ref{t_k lower}). Set
\begin{align*}
    A_k := \{ x \in B_R^N \,\,|\,\, t_k u_k (x) \ge T_\ep \}, \,\,D_k := B_R^N \setminus A_k.
\end{align*}
Note that $A_k = B^N_{R(k)}$, where $R(k) := \( R^{-\frac{N-p}{p-1}} + \frac{(N-p) T_\ep }{(p-1) t_k } k^{\frac{1}{p}} \w_p^{-\frac{1}{p(p-1)}} \w_N^{\frac{1}{p-1}} \)^{-\frac{p-1}{N-p}} \to 0$ as $k \to \infty$.
Then we see that 
\begin{align*}
    t_k^p &= \int_{A_k} f(t_k u_k) t_k u_k V(|x|) \,dx + \int_{D_k}  f(t_k u_k) t_k u_k V(|x|) \,dx\\
    &\ge (\beta -\ep) C_V \int_{A_k} e^{\alpha_0 t_k^{p'} u_k^{p'}} V_p(|x|) \,dx + \int_{D_k}  f(t_k u_k) t_k u_k V(|x|) \,dx\\
    &= (\beta -\ep) C_V \int_{B_R^N} e^{\alpha_0 t_k^{p'} u_k^{p'}} V_p(|x|) \,dx - (\beta -\ep) C_V \int_{D_k} e^{\alpha_0 t_k^{p'} u_k^{p'}} V_p(|x|) \,dx + \int_{D_k}  f(t_k u_k) t_k u_k V(|x|) \,dx\\
    &=: I_1(k) - I_2(k) + I_3(k).
\end{align*}
By using $t_k u_k \to 0$ a.e. in $B_R^N$ and the Lebesgue dominated convergence theorem, we have
\begin{align*}
    I_3(k) &\to 0,\\
    I_2(k) &\to (\beta -\ep) C_V \int_{B_R^N} V_p (|x|) \,dx = (\beta -\ep) C_V \frac{\w_p}{p}.
\end{align*}
For $I_1(k)$, we have
\begin{align*}
    \lim_{k \to \infty} I_1(k) 
    &\ge  (\beta -\ep) C_V \lim_{k \to \infty} \int_{B_R^N} e^{\alpha_p u_k^{p'}} V_p(|x|) \,dx\\
    &= (\beta -\ep) C_V \left[ \frac{\w_p}{p} + \lim_{k \to \infty} \int_{B_R^N \setminus B_{r_k}^N} e^{\alpha_p u_k^{p'}} V_p(|x|) \,dx \right]    
\end{align*}
and by using the change of variables $\( \frac{\w_p}{\w_N} \)^{\frac{1}{p-1}} \( \frac{p-1}{N-p} \) \( r^{-\frac{N-p}{p-1}} - R^{-\frac{N-p}{p-1}} \) = kt$, 
\begin{align*}
    &\int_{B_R^N \setminus B_{r_k}^N} e^{\alpha_p u_k^{p'}} V_p(|x|) \,dx\\
    &= \w_N  \int_{r_k}^R \exp \Bigg\{ pk^{1-p'} \( \frac{p-1}{N-p} \)^{p'} \( \frac{\w_p}{\w_N} \)^{\frac{p'}{p-1}}  \( r^{-\frac{N-p}{p-1}} - R^{-\frac{N-p}{p-1}} \)^{p'}\\
    &\hspace{6em}- p \frac{p-1}{N-p} \( \frac{\w_p}{\w_N} \)^{\frac{1}{p-1}}  \( r^{-\frac{N-p}{p-1}} - R^{-\frac{N-p}{p-1}} \) \Bigg\} \(\frac{\w_p}{\w_N}\)^{p'} r^{(N-1)(1-p')} \,dr \\
    &= \frac{\w_p}{p} pk \int_0^1 e^{pk (t^{p'} -t)} \,dt \to \frac{\w_p L_p}{p} \,\,(k \to \infty).
\end{align*}
Therefore, we get
\begin{align*}
    \frac{\alpha_p}{\alpha_0} = \lim_{k \to \infty} t_k^{p'} 
    \ge \left[ (\beta -\ep) C_V L_p \frac{\w_p}{p} \right]^{\frac{1}{p-1}}. 
\end{align*}
Since $\ep >0$ is arbitrary, we get $\beta \le \frac{p }{ C_V L_p \w_p} \(\frac{\alpha_p}{\alpha_0}\)^{p-1} = \frac{p^p}{ \alpha_0^{p-1} C_V L_p}$ which contradicts (A9). Hence we get (\ref{Claim d<c}). 
\end{proof}


By using above lemmas, we show Theorems.

\begin{proof}{\it ( Theorem \ref{Thm subcritical})}
Theorem \ref{Thm subcritical} follows from the mountain pass theorem, Lemma \ref{Lem MPG} and Lemma \ref{Lem PS} (I). 
\end{proof}

\begin{proof}{\it ( Theorem \ref{Thm critical})}
Theorem \ref{Thm critical} follows from the mountain pass theorem, Remark \ref{Rem A5}, Lemma \ref{Lem MPG}, Lemma \ref{Lem PS} (II) and Lemma \ref{Lem d<c}. 
\end{proof}


\section*{Acknowledgment}
The first author (M.I.) was supported by JSPS KAKENHI Grant-in-Aid for Young Scientists Research, No. JP19K14581 and JST CREST, No. JPMJCR1913. 
The second author (M.S.) was supported by JSPS KAKENHI Early-Career Scientists, No. 23K13001. 
The third author (K.T.) was supported by Grant for Basic Science Research Projects from The Sumitomo Foundation, No.210788.

This work was partly supported by Osaka Central Advanced Mathematical Institute: MEXT Joint Usage/Research Center on Mathematics and Theoretical Physics JPMXP0619217849.


\end{document}